\documentclass[notitlepage,11pt,reqno]{amsart}
\usepackage[foot]{amsaddr}
\usepackage[numbers,sort]{natbib}

\usepackage{amssymb,nicefrac,bm,upgreek,mathtools,verbatim}
\usepackage[final]{hyperref}
\usepackage[mathscr]{eucal}
\usepackage{dsfont}

\usepackage{color}
\definecolor{dmagenta}{rgb}{.4,.1,.5}
\definecolor{dblue}{rgb}{.0,.0,.5}
\definecolor{mblue}{rgb}{.0,.0,.8}
\definecolor{ddblue}{rgb}{.0,.0,.4}
\definecolor{dred}{rgb}{.6,.0,.0}
\definecolor{dgreen}{rgb}{.0,.5,.0}
\definecolor{Eeom}{rgb}{.0,.0,.5}

\usepackage[normalem]{ulem}

\usepackage[margin=1in]{geometry}
\allowdisplaybreaks

\newtheorem{lemma}{Lemma}[section]
\newtheorem{theorem}{Theorem}[section]
\newtheorem{proposition}{Proposition}[section]
\newtheorem{corollary}{Corollary}[section]

\theoremstyle{definition}
\newtheorem{definition}{Definition}[section]

\theoremstyle{remark}
\newtheorem{example}{Example}[section]
\newtheorem{remark}{Remark}[section]

\numberwithin{equation}{section}

\hypersetup{
  colorlinks=true,
  citecolor=dblue,
  linkcolor=dblue,
  frenchlinks=false,
  pdfborder={0 0 0},
  naturalnames=false,
  hypertexnames=false,
  breaklinks}
\usepackage[capitalize,nameinlink]{cleveref}

\crefname{section}{Section}{Sections}
\crefname{subsection}{Subsection}{Subsections}
\crefname{condition}{Condition}{Conditions}
\crefname{hypothesis}{Hypothesis}{Conditions}
\crefname{assumption}{Assumption}{Assumptions}
\crefname{lemma}{Lemma}{Lemmas}
\crefname{claim}{Claim}{Claims}

\Crefname{figure}{Figure}{Figures}

\crefformat{equation}{\textup{#2(#1)#3}}
\crefrangeformat{equation}{\textup{#3(#1)#4--#5(#2)#6}}
\crefmultiformat{equation}{\textup{#2(#1)#3}}{ and \textup{#2(#1)#3}}
{, \textup{#2(#1)#3}}{, and \textup{#2(#1)#3}}
\crefrangemultiformat{equation}{\textup{#3(#1)#4--#5(#2)#6}}%
{ and \textup{#3(#1)#4--#5(#2)#6}}{, \textup{#3(#1)#4--#5(#2)#6}}%
{, and \textup{#3(#1)#4--#5(#2)#6}}

\Crefformat{equation}{#2Equation~\textup{(#1)}#3}
\Crefrangeformat{equation}{Equations~\textup{#3(#1)#4--#5(#2)#6}}
\Crefmultiformat{equation}{Equations~\textup{#2(#1)#3}}{ and \textup{#2(#1)#3}}
{, \textup{#2(#1)#3}}{, and \textup{#2(#1)#3}}
\Crefrangemultiformat{equation}{Equations~\textup{#3(#1)#4--#5(#2)#6}}%
{ and \textup{#3(#1)#4--#5(#2)#6}}{, \textup{#3(#1)#4--#5(#2)#6}}%
{, and \textup{#3(#1)#4--#5(#2)#6}}

\crefdefaultlabelformat{#2\textup{#1}#3}

\makeatletter
\DeclareRobustCommand\widecheck[1]{{\mathpalette\@widecheck{#1}}}
\def\@widecheck#1#2{%
    \setbox\z@\hbox{\m@th$#1#2$}%
    \setbox\tw@\hbox{\m@th$#1%
       \widehat{%
          \vrule\@width\z@\@height\ht\z@
          \vrule\@height\z@\@width\wd\z@}$}%
    \dp\tw@-\ht\z@
    \@tempdima\ht\z@ \advance\@tempdima2\ht\tw@ \divide\@tempdima\thr@@
    \setbox\tw@\hbox{%
       \raise\@tempdima\hbox{\scalebox{1}[-1]{\lower\@tempdima\box
\tw@}}}%
    {\ooalign{\box\tw@ \cr \box\z@}}}
\makeatother

\newcommand{\df}{\coloneqq}
\DeclareMathOperator{\Exp}{\mathbb{E}} 
\newcommand{\D}{\mathrm{d}}          
\newcommand{\RR}{\mathbb{R}}         
\newcommand{\RN}{{\mathbb{R}^N}}       
\newcommand{\Ind}{\mathds{1}}            
\newcommand{\Uadm}{\mathfrak{U}}     

\newcommand{\Sob}{\mathscr{W}}       
\newcommand{\Sobl}{\mathscr{W}_{\mathrm{loc}}}  

\newcommand{\sB}{\mathscr{B}}    
\newcommand{\cC}{\mathcal{C}}     

\newcommand{\cI}{\mathcal{I}}    
\newcommand{\cM}{\mathcal{M}}
   
\newcommand{\cS}{\mathcal{S}}   
    
\newcommand{\abs}[1]{\lvert#1\rvert}
\newcommand{\norm}[1]{\lVert#1\rVert}

\newcommand{\plamplus}{\lambda^{\prime,+}_{1}}
\newcommand{\pplamplus}{\lambda^{\prime\prime,+}_{1}}
\newcommand{\plamminus}{\lambda^{\prime,-}_{1}}
\newcommand{\pplamminus}{\lambda^{\prime\prime,-}_{1}}

\DeclareMathOperator*{\supp}{support}

\DeclareMathOperator{\trace}{trace}

\begin{document}

\title[Generalized principal eigenvalues of nonlinear elliptic operators in $\RN$]%
{Generalized principal eigenvalues of convex nonlinear elliptic operators in $\RN$}

\author[Anup Biswas]{Anup Biswas}
\address{$^\dag$ Department of Mathematics,
Indian Institute of Science Education and Research,
Dr. Homi Bhabha Road, Pune 411008, India}
\email{anup@iiserpune.ac.in, prasunroychowdhury1994@gmail.com}

\author[Prasun Roychowdhury]{Prasun Roychowdhury}

\date{}

\begin{abstract}
We study the generalized eigenvalue problem in $\RN$ for a general convex nonlinear  elliptic operator 
which is locally elliptic and positively homogeneous.
Generalizing Berestycki and Rossi \cite{BR15} we consider three different notions of generalized eigenvalues
and compare them. We also discuss the maximum principles and uniqueness of principal eigenfunctions.
\end{abstract}



\maketitle

\section{Introduction}
This article contributes to the study of eigenvalue problem of the form  
$$F(D^2 \psi, D\psi, \psi, x)\,=\, \lambda\psi\quad \text{in}\; \RN,$$
where $F$ is a  fully nonlinear, convex, positively $1$-homogeneous elliptic operator with measurable coefficients. We establish the
existence of half (or demi) eigenvalues and characterize the set of all eigenvalues with positive and negative
eigenfunctions. This generalizes a recent work of Berestycki and Rossi \cite{BR15} which considers
linear elliptic operators. We also derive necessary and sufficient conditions for the validity maximum
principles in $\RN$ and discuss the uniqueness of principal eigenfunctions.

It has long been known that certain types of positively homogeneous operators possess 
\textit{two principal eigenvalues} (one corresponds to a positive eigenfunction and the other one
corresponds to a negative eigenfunction).  In fact, this first appeared in the work of Pucci \cite{P66}
who computed these eigenvalues explicitly for certain extremal operators in the unit ball. Later
it also appeared in a work of Berestycki \cite{B77} while studying the bifurcation phenomenon for some nonlinear Sturm-Liouville problem and Berestycki referred them as \textit{half eigenvalues}.
In connection to this work of Berestycki,
 Lions used a stochastic control approach in \cite{PL83} to characterize
these eigenvalues (he called it demi-eigenvalues) for operators which are the supremum of linear operators with $\cC^{1,1}$-coefficients, and
relate them to certain bifurcation problem. In their seminal work \cite{BNV} Berestycki, Nirenberg and Varadhan
introduced the notion of Dirichlet {\it generalized principal eigenvalue} for linear operators in {\it non-smooth} bounded domains and also established a deep connection between sign of the principal eigenvalue
and validity of maximum principles. This work serves as a founding stone of the modern eigentheory 
and has been used to study eigenvalue problems for general nonlinear operators, including degenerate ones. 
We are in particular, attracted by the works \cite{Arm09,Arm09a,BCPR,BD06,B99,BEQ,II12,IY06,SP09,QS08}.
We owe much to the work of Quass and Sirakov \cite{QS08} who study the Dirichlet principal eigenvalue 
problem for convex, fully nonlinear operators in bounded domains. 

All the above mentioned works deal with bounded domains. It is then natural to ask how the eigenteory
changes for unbounded domains. In fact, the necessity for studying eigenvalue problems in $\RN$ 
becomes important to understand the existence and uniqueness of solutions for certain semilinear elliptic
operators. See for instance, the discussion in \cite{BR06,BR15} and references therein. Principal eigenvalue
is a key ingredient to find the rate functional for the large deviation estimate
of empirical measures of diffusions \cite{DV75,DV76,KS06}. Recently, eigenvalue problems in $\RN$
have got much attention due to its application in the theory of risk-sensitive controls 
\cite{ABS19,ABBS,AB20} (some discussions are left to \cref{S-motiv}).  Our present work is motivated
by a recent study of Berestycki and Rossi \cite{BR15} where the authors consider  non-degenerate
linear elliptic operators and develop an eigentheory for unbounded domains. Monotonicity property
of the principal eigenvalue (with respect to the potentials) in $\RN$
 and its relation with the stability property of the {\it twisted process}
is established in \cite{ABS19}. The paper \cite{AB20} considers a class of semilinear elliptic operators
and obtains a variational representation of the principal eigenvalue under the assumption of 
geometric stability. The chief goal of this article is to develop an eigentheory for fully nonlinear
positively homogeneous operators. Though the results of this article are obtained in the whole space
$\RN$, one can mimic the arguments for any unbounded domains  (see \cref{R3.1} for more details).

The rest of the article is organized as follows: In the next section we introduce our model and state
our main results. We also motivate the model by providing a discussion in \cref{S-motiv}. Proofs of the
main results are given in \cref{S-proof}.

\section{Statement of main results}\label{S-main}
In this section we introduce our model and state the main results. We also provide a motivation
in \cref{S-motiv} for considering these eigenvalue problems.

\subsection{Model and assumptions}\label{S-model}
Let $\lambda, \Lambda: \RN\to (0, \infty)$ be two locally bounded functions with the property that
for any compact set $K\subset \RN$ we have
$$0\,<\, \inf_{x\in K} \lambda(x)\,\leq\, \sup_{x\in K}\Lambda(x)\,<\, \infty.$$
Choosing $K=\{x\}$ it follows from above that $0<\lambda(x)\leq\Lambda(x)$ for all $x\in\RN$. These two 
functions will be treated as the bounds of the ellipticity constants at point $x$. By $\cS_N$ we denote 
the set of all $N\times N$ real symmetric matrices. The extremal Pucci operators corresponding to
$\lambda, \Lambda$ are defined as follows. For $M\in\cS_N$ the extremal operators at $x\in\RN$ are
defined to be
\begin{align*}
\cM^+_{\lambda, \Lambda}(x, M)&=\,\sup_{\lambda(x)I\leq A\leq \Lambda(x)I} \trace(AM)
= \Lambda(x)\sum_{\beta_i\geq 0} \beta_i + \lambda(x) \sum_{\beta_i<0} \beta_i,
\\
\cM^-_{\lambda, \Lambda}(x, M)&=\,\inf_{\lambda(x)I\leq A\leq \Lambda(x)I} \trace(AM)
= \lambda(x)\sum_{\beta_i\geq 0} \beta_i + \Lambda(x) \sum_{\beta_i<0} \beta_i,
\end{align*}
where $\beta_1, \ldots, \beta_n,$ denote the eigenvalues of the matrix $M$.

Our operator $F$ is a Borel measurable function 
$F:\cS_N\times\RN\times\RR\times\RN\to \RR,$
with the following properties:
\begin{itemize}
\item[(H1)] $F$ is positively $1$-homogeneous in the variables $(M, p, u)\in\cS_N\times\RN\times\RR$ i.e., 
for every $t>0$ we have we have
$$F(tM, tp, tu, x)\, =\, t F(M, p, u, x)\quad \text{for all}\; x\in\RN\,.$$
In particular, $F(0,0,0,x)\equiv 0$.

\item[(H2)] $F$ is convex in the variables $(M, p, u)\in\cS_N\times\RN\times\RR$.
\item[(H3)] There exist locally bounded functions $\gamma, \delta:\RN\to [0, \infty)$ satisfying
\begin{align*}
	\cM_{\lambda,\Lambda}^-(x, M-N)-\gamma(x)|p-q|-\delta(x)|u-v| & \leq F(M,p,u,x)-F(N,q,v,x)
	\\
	&\leq\,\mathcal{M}_{\lambda,\Lambda}^+(x, M-N) + \gamma(x)|p-q| + \delta(x)|u-v|,
	\end{align*}
for all $M, N\in\cS_N$, $p, q\in\RN, u, v\in\RR$ and $x\in\RN$.
\item[(H4)] The function $(M, x)\in \cS_N\times\RN\mapsto F(M, 0, 0, x)$ is continuous.
\end{itemize}
Throughout this article we assume the conditions (H1)--(H4) without any further mention.
Also, observe that due to our hypotheses the operator $F$ satisfies the conditions in \cite{QS08} which
studies the Dirichlet eigenvalue problem for $F$ in bounded domains. Therefore the results of \cite{QS08}
holds for $F$ in smooth bounded domains.

Let us now define the principal eigenvalues of $F$ in a smooth domain $\Omega\subset\RN$,  possibly unbounded.
For any real number $\lambda$ we define the following sets
	\begin{align*}
	\Psi^+(F,\Omega,\lambda) &=\{\psi\in \Sobl^{2, N}(\Omega)\;:\; 
	F(D^2\psi,D\psi,\psi,x) + \lambda\psi\leq 0\; \text{and }\psi>0\; \text{ in }\; \Omega\}\,,
	\\
	\Psi^-(F,\Omega,\lambda)&=\{\psi\in \Sobl^{2, N}(\Omega) \;:\; F(D^2\psi,D\psi,\psi,x)+\lambda\psi\geq 0
	\; \text{ and }\; \psi<0\; \text{in}\; \Omega\}\,.
	\end{align*}
By sub or supersolution we always mean $L^N$-strong solution.
The (half) eigenvalues are defined to be
\begin{align*}
	\lambda_1^+(F,\Omega)&=\text{sup}\{\lambda\in\mathbb{R}\;:\;\Psi^+(F,\Omega,\lambda)\neq\emptyset\}\,,
	\\
	\lambda_1^-(F,\Omega)&=\text{sup}\{\lambda\in\mathbb{R}\;:\;\Psi^-(F,\Omega,\lambda)\neq\emptyset\}\,.
	\end{align*}
Using the convexity of $F$ and \cite[Proposition~4.2]{QS08} it follows that $\lambda_1^+(F,\Omega)\leq
\lambda_1^-(F,\Omega)<\infty$. For $F$ linear we also have $\lambda_1^+(F,\Omega)= \lambda_1^-(F,\Omega)$.
In this article we would be interested in the case $\Omega=\RN$ and for
notational economy we denote $\lambda_1^{\pm}(F,\RN)=\lambda_1^{\pm}(F)$.

\begin{remark}
We can replace the $L^N$-strong super and subsolutions in $\Psi^{\pm}(F,\Omega,\lambda)$ by
$L^N$-viscosity super and subsolutions, respectively.
\end{remark}
\subsection{Main results}\label{S-result}
We now state our main results. Most of the results obtained here are generalization of its linear counterpart
in \cite{BR15}.
 Recall from \cite[Theorem~1.4]{BR15} that for $F$ linear and $\lambda\in (-\infty, \lambda_1(F)]$ there
exists a positive function $\varphi\in\Sobl^{2, p}(\RN), p>N,$ satisfying 
$F(D^2\varphi, D\varphi, \varphi,x)+\lambda\varphi=0$ in $\RN$. Thus there is a continuum of eigenvalues with 
the largest one being the principal eigenvalue. This leads us to the following sets of eigenvalues.

\begin{definition}\label{D2.1}
We say $\lambda\in\RR$ is an {\it eigenvalue with a positive eigenfunction} if there exists
 $\phi\in \Sobl^{2,p}(\RN),  p> N$, such that 
		\begin{align*}
		F(D^2\phi,D\phi,\phi,x)=-\lambda\phi\quad  \text{in}\; \RN\,,\quad \text{and}\quad  \phi>0 
		\quad \text{in}\; \RN.
		\end{align*} 
We denote the collection of all eigenvalues with positive eigenfunctions by $\mathcal{E}^+$.
Analogously, we define $\mathcal{E}^-$ as the collection of all 
{\it eigenvalues with negative eigenfunctions}.
\end{definition}
Our first result generalizes \cite[Theorem~1.4]{BR15}.
\begin{theorem}\label{T2.1}
We have $\mathcal{E}_+=(-\infty,\lambda_1^+(F)]$ and $\mathcal{E}_-=(-\infty, \lambda_1^-(F)]$.
\end{theorem}
It is well known that for bounded domains it is also possible to define principal eigenvalues through sub-solutions (cf. \cite[Theorem~1.2]{QS08}). However, this situation is bit different for unbounded domains.
To explain we introduce the following quantities.

\begin{align*}
\lambda_1^{\prime,+}(F) &\df \inf\{\lambda\in \mathbb{R}\;:\; \exists\, \psi\in 
\Sobl^{2,N}(\RN)\cap\text{L}^\infty(\RN),\psi>0, F(D^2\psi,D\psi,\psi,x) + \lambda\psi\geq 0\; \text{in }\, \RN\},
\\
\lambda_1^{\prime,-}(F)&\df\inf\{\lambda\in \mathbb{R}\;:\;\exists\, \psi\in
 \Sobl^{2,N}(\RN)\cap\text{L}^\infty(\RN),\psi<0, F(D^2\psi,D\psi,\psi,x) + \lambda\psi\leq 0\; \text{in }\, \RN\},
\end{align*}
and
\begin{align*}
\lambda_1^{\prime\prime,+}(F)&\df \sup\{\lambda\in \mathbb{R}\;:\; \exists\, \psi\in
 \Sobl^{2,N}(\RN),\;\inf_{\RN}\psi>0,\; F(D^2\psi,D\psi,\psi,x) + \lambda\psi\leq 0\; \text{in }\, \RN\}\,,
\\
\lambda_1^{\prime\prime,-}(F)&\df\sup\{\lambda\in \mathbb{R}\;:\;\exists\, \psi\in
 \Sobl^{2,N}(\RN),\;\sup_{\RN}\psi<0,\; F(D^2\psi,D\psi,\psi,x) + \lambda\psi\geq 0\; \text{in }\, \RN\}\,.
\end{align*}
We remark that in case of bounded domains one has 
$$\lambda_1(F, \Omega)=\plamplus(F, \Omega)=\pplamplus(F, \Omega)$$ and
$$\lambda_1(F, \Omega)=\plamminus(F, \Omega)=\pplamminus(F, \Omega),$$ provided we required the subsolution
(supersolution) to vanish on $\partial\Omega$ in the definition of $\plamplus$ ($\plamminus$, resp.)
(cf. \cite{QS08}). But the same might fail to hold in unbounded domains (counter-example in \cite[p.~201]{BR06}).
However, we could prove the following relation which generalizes \cite[Theorem~1.7]{BR15}.
\begin{theorem}\label{T2.2}
The following hold.
\begin{itemize}
\item[(i)] We have $\plamplus(F)\leq\lambda_1^+(F)$ and $\plamminus(F)\leq\lambda_1^-(F)$.
\item[(ii)] Suppose that 
\begin{align}\label{ET2.2A}
	\sup_{\RN}\, \delta(x)<\infty, \quad \limsup_{|x|\rightarrow\infty}\ \frac{\gamma(x)}{|x|}<\infty,
	\quad \text{and}\quad \limsup_{|x|\rightarrow\infty}\ \frac{\Lambda(x)}{\abs{x}^2}\,<\,\infty\,.
	\end{align}
Then we have $\pplamplus(F)\leq \plamplus(F)$ and $\pplamminus(F)\leq \plamminus(F)$.
\end{itemize}
\end{theorem}

In view of \cref{T2.2} we see that $\pplamplus(F)\leq\plamplus(F)\leq\lambda_1^+(F)$ and
$\pplamminus(F)\leq\plamminus(F)\leq\lambda_1^-(F)$, provided \eqref{ET2.2A} holds. Again, due to
the convexity of $F$ we have $\lambda_1^+(F)\leq \lambda_1^-(F)$. One might wonder if there is any
natural relation between ``plus" and ``minus" eigenvalues. We now argue that this might not be possible,
in general.
If we consider $F$ to be linear then
we have $\lambda^{\cdot, +}_1(F)=\lambda^{\cdot, -}_1(F)$, and therefore if \eqref{ET2.2A} holds, then 
$\lambda^{+}_1(F)\geq\pplamminus(F)$, by \cref{T2.2}. We now produce an example where the reverse 
inequality holds.
\begin{example}
Consider two linear elliptic operators of the form 
	$$L_\alpha u\,=\,\Delta u + b_\alpha(x)\cdot D u +c_\alpha(x) u,$$
	for $\alpha\in\{1,2\}$ with the properties that  
$$\lambda_{1}^{\prime\prime}(L_2,\RN)\,>\,\lambda_{1}^{\prime\prime}(L_1,\RN)\quad
 \text{and}\quad \lambda_{1}^{\prime\prime}(L_1,\RN)=\lambda_{1}^{\prime}(L_1,\RN)=\lambda_{1}(L,\RN).$$
Now define a nonlinear operator 
$$F(D^2u,Du,u,x)\df\Delta u+\max_{\alpha\in\{1,2\}}\{b_\alpha(x)\cdot Du\}+
c_\alpha(x) u\}.$$
It is then easily seen that	
$$\lambda_{1}^{''-}(F) \,\geq\, \max\{\lambda_{1}^{\prime\prime}(L_1,\RN),
\lambda_{1}^{\prime\prime}(L_2,\RN)\}\,,$$
 and 
$$\lambda_{1}^{+}(F) \,\leq\, \min\{\lambda_{1}(L_1,\RN),\lambda_{1}(L_2,\RN)\}.$$
Combining we obtain	
$$\pplamminus(F) \geq \lambda_{1}^{\prime\prime}(L_2)>\lambda_{1}^{\prime\prime}(L_1,\RN)
=\lambda_{1}(L_1,\RN)\geq \lambda_{1}^{+}(F).$$
\end{example}
Next we list a few class of operators for which these three eigenvalues coincide (compare
them with \cite[Theorem~1.9]{BR15}). We only provide the result for ``plus" eigenvalues and the 
analogous result for ``minus" eigenvalues are easy to guess.
\begin{theorem}\label{T2.3}
 The equality $\lambda_{1}^{+}(F)=\lambda_{1}^{\prime\prime,+}(F)$ holds in each of the following cases:
 \begin{itemize}
\item[(i)] $F=\tilde{F}+\tilde\gamma(x),$ where $\tilde{F}$ is a nonlinear operator with an additional property $\lambda_{1}^{+}(\tilde{F},\RN)=\lambda_{1}^{\prime\prime,+}(\tilde{F},\RN)$, and $\tilde\gamma\in L^\infty(\RN)$ is a non-negative function satisfying
$\lim_{|x|\rightarrow\infty}\tilde\gamma(x)=0$.
\item[(ii)] $\lambda_{1}^{+}(F)\,\leq\,-\limsup_{|x|\rightarrow\infty}\ F(0,0,1,x)$.
\item[(iii)] Assume that $\lambda_0\leq\lambda(x)\leq\Lambda(x)\leq\Lambda_0$ for all
$x\in\RN$, $\lim_{\abs{x}\to\infty}\gamma(x)=0$ and for all $r>0$ and all $\beta$ such that 
$\beta<\,\limsup_{\abs{x}\to\infty} F(0,0,1,x),$ there exists $\sB_r(x_0)$ satisfying $\inf_{\sB_r(x_0)} F(0,0,1,x)>\beta\,.$
\item[(iv)] There exists a $V\in\cC^2(\RN)$ with $\inf_{\RN} V>0$ and 
$$F(D^2V, DV, V, x)\leq -\lambda_1^+(F) V\quad \text{for all}\; x\in\sB^c,$$
for some ball $\sB$.
\end{itemize}
\end{theorem}

Now we turn our attention towards maximum principles. It was observed in the seminal work of Berestycki,
Nirenberg and Varadhan \cite{BNV} that the sign of the principle eigenvalue determines the
validity of maximum principles in bounded domains. Extension of this result for nonlinear operators are obtained by 
Quaas and Sirakov \cite{QS08} and Armstrong \cite{Arm09}.
Further generalization in smooth bounded domains 
for a class of degenerate, nonlinear elliptic operators are obtained by Berestycki et.\ al.\ \cite{BCPR},
Birindelli and Demengel \cite{BD06}.
Recently,
Berestycki and Rossi \cite{BR15} establish the maximum principles in unbounded domains for linear elliptic 
operators. Here we extend their results to our nonlinear setting. 
\begin{definition}[Maximum principles]
We say that the operator $F$ satisfies $\beta^{+}$-MP with respect to a positive function $\beta$
if for any function 
$u\in\Sobl^{2, N}(\RN)$ satisfying
$$F(D^2u,Du,u,x)\;\geq\; 0 \quad \text{in }\; \RN,\quad \text{and} \quad 
\sup_{\RN}\ \frac{u}{\beta}\,<\,\infty\,,$$
we have $u\leq 0$ in $\RN$. For $\beta=1$, we simply mention this property as $+$MP.

We say that the operator $F$ satisfies $\beta^{-}$-MP with respect to a negative function $\beta$ if for any function 
$u\in\Sobl^{2, N}(\RN)$ satisfying
$$F(D^2u,Du,u,x)\;\leq\; 0 \quad \text{in }\; \RN,\quad \text{and} \quad 
\sup_{\RN}\ \frac{u}{\beta}\,<\,\infty\,,$$
we have $u\geq 0$ in $\RN$. For $\beta=-1$, we simply mention this property as $-$MP.
\end{definition}
Note that $\beta\equiv 1$ corresponds to the well known maximum principle.
We would be interested in a function $\beta:\RN\rightarrow (0, \infty)$  which satisfies either
\begin{align}\label{E2.3}
	\exists\, \sigma>0,\quad \limsup_{|x|\rightarrow\infty}\ \beta(x)|x|^{-\sigma}=0\,,
	\end{align}
	or
	\begin{align}\label{E2.4}
	\exists\;\sigma>0,\quad \limsup_{|x|\rightarrow\infty}\ \beta(x)\exp(-\sigma|x|)=0\,.
\end{align}
Generalizing \cite[Definition~1.2]{BR15} we now consider the following quantities.
\begin{definition} Given a positive function $\beta:\RN\to \RR$, we define
\begin{align*}
\lambda_{\beta}^{\prime\prime,+}(F)&\df\sup \{\lambda\in \mathbb{R}\;:\;\exists\; \psi\in 
\Sobl^{2,N}(\RN),\ \psi\geq\beta, \ F(D^2\psi,D\psi,\psi,x)+\lambda\psi\leq 0\; \text{in }\, \RN\}\,,
\\
\lambda_{\beta}^{\prime\prime,-}(F) &\df\sup \{\lambda\in \mathbb{R}\; :\;\exists\, \psi
\in\Sobl^{2,N}(\RN),\ \psi\leq-\beta, \ F(D^2\psi,D\psi,\psi,x)+\lambda\psi\geq 0\; \text{in }\, \RN\}\,.
\end{align*}
\end{definition}
Our maximum principles would be established under the following growth conditions on
the coefficients.
	\begin{align}\label{E2.5}
	\sup_{\RN}\ \delta(x)<\infty, \quad \limsup_{|x|\rightarrow\infty}\ \frac{\gamma(x)}{|x|}<\infty,\quad \text{and}\quad \limsup_{|x|\rightarrow\infty}\ \frac{\Lambda(x)}{|x|^2}\,<\,\infty\,.
	\end{align}
	 
	\begin{align}\label{E2.6}
	\sup_{\RN}\ \delta(x)<\infty, \quad \sup_{\RN}\ \gamma(x)<\infty,\quad \text{and}
	 \quad \sup{\RN}\ \Lambda(x)<\infty\,.
	\end{align}
Next we state our maximum principle

\begin{theorem}\label{T2.4}
Suppose that either \eqref{E2.3} and \eqref{E2.5} or \eqref{E2.4} and \eqref{E2.6} hold. Then the following hold:
\begin{itemize}
\item[(i)]The operator $F$ satisfies $\beta^+$-MP in $\RN$ if $\lambda_{\beta}^{\prime\prime,+}(F)>0$.
\item[(ii)] The operator $F$ satisfies $(-\beta)^-$-MP in $\RN$ if $\lambda_{\beta}^{\prime\prime,-}(F)>0$.
\end{itemize}
\end{theorem}
As a consequence of \cref{T2.4} we obtain the following corollaries.
\begin{corollary}
Suppose that either \eqref{E2.5} or \eqref{E2.6} holds. Then we have
\begin{itemize}
\item[(i)] The operator $F$ satisfy $+$MP in $\RN$ if $\pplamplus(F)>0$.
\item[(ii)] The operator $F$ satisfy $-$MP in $\RN$ if $\pplamminus(F)>0$.
\item[(iii)] Suppose that $\pplamplus(F)>0$ (therefore, $\pplamminus(F)>0$).
Let $u\in\Sobl^{2,N}(\RN)\cap\L^\infty(\RN)$  satisfy $F(D^2u,Du,u,x)= 0$ in $\RN$. Then $u\equiv 0$.
\end{itemize}
\end{corollary}
\begin{corollary}
Suppose that $F$ satisfies $\beta^+$-MP. Let $u, v\in\Sobl^{2,N}(\RN)$ be such that
$$F(D^2 u, Du, u, x)\,\geq 0, \quad F(D^2 v, Dv, v, x)\,\leq 0 \; \text{in}\; \RN, \quad \text{and}\quad 
\sup_{\RN}\frac{u-v}{\beta}\,<\,\infty\,. $$
Then we have $u\leq v$ in $\RN$.
\end{corollary}
\begin{proof}
Denote by $w=u-v$. By using the convexity of $F$ it follows that 
$$F(D^2 w, Dw, w, x)\geq F(D^2 u, Du, u, x)-F(D^2 v, Dv, v, x)\geq 0\quad \text{in}\; \RN.$$
Hence the result follows from $\beta^+$-MP.
\end{proof}
Generalizing $\plamplus(F)$ and $\plamminus(F)$ we define the following quantities. Let
$\beta$ be a positive valued function and
\begin{align*}
\lambda_{\beta}^{\prime,+}(F)\df\inf \{\lambda\in \mathbb{R}\;:\;\exists\;\psi\in
\Sobl^{2, N}(\RN),\; \beta\geq\psi>0,\;  F(D^2\psi,D\psi,\psi,x) +\lambda\psi\geq 0\; \text{in }\; \RN\}\,,
\end{align*}
and
\begin{align*}
\lambda_{\beta}^{\prime,-}(F) :=\inf \{\lambda\in \mathbb{R}\;:\;\exists\; \psi\in
\Sobl^{2, N}(\RN),\; -\beta\leq\psi<0,\;  F(D^2\psi,D\psi,\psi,x)+\lambda\psi\leq 0\; \text{in}\; \RN\}\,.
\end{align*}
As a necessary condition for the validity of maximum principles we deduce the following.
\begin{theorem}\label{T2.5}
The following hold.
\begin{itemize}
\item[(i)] If $F$ satisfies the $\beta^+$-MP then $\lambda_{\beta}^{\prime,+}(F)\geq 0$. In particular, 
if $F$ satisfies $+$MP then we have $\lambda_{1}^{\prime,+}(F)\geq 0$.
\item[(ii)] If $F$ satisfies the $(-\beta)^-$-MP then $\lambda_{\beta}^{\prime,-}(F)\geq 0$.
In particular, if $F$ satisfies the $-$MP then we have $\lambda_{1}^{\prime,-}(F)\geq 0$.
\end{itemize}
\end{theorem}
Finally, we discuss about simplicity of the principal eigenvalues. For linear $F$ uniqueness
of principal eigenfunctions can be established imposing {\it Agmon's minimal growth condition at
 infinity} \cite[Definition~8.2]{BR15} on the eigenfunctions.  But such criterion does not seem to
 work well for nonlinear $F$.
Recently, in \cite[Theorem~2.1]{ABG} it is shown that Agmon's minimal growth criterion
 is equivalent to {\it monotonicity of the principal eigenvalue on the right}. Our next result
 establish simplicity of principal eigenvalue under certain monotonicity
 condition of principal eigenvalue {\it at infinity}.
\begin{theorem}\label{T2.6}
Suppose that there exists a positive $V\in\Sobl^{2, N}(\RN)$ satisfying
\begin{equation}\label{ET2.6A}
F(D^2V, DV, V, x)\,\leq\, -(\lambda_1^+(F)+\varepsilon) V\quad \text{for all}\; x\in K^c,
\end{equation}
for some compact ball $K$ and $\varepsilon>0$.
Then $\lambda^+_1(F)$ is simple i.e. the positive eigenfunction is unique upto a multiplicative constant.
\end{theorem}
We remark that \eqref{ET2.6A} is equivalent to 
$$\lambda^+_1(F)\,<\, \lim_{r\to\infty}\lambda^+_1(F, \bar\sB^c_r).$$
Our next result is about simplicity of $\lambda^{-}_1(F)$.
 \begin{theorem}\label{T2.7}
Suppose that there exists a positive $V\in\Sobl^{2, N}(\RN)$ satisfying
\begin{equation}\label{ET2.7B}
F(D^2V, DV, V, x)\leq -(\lambda_1^-(F)+\varepsilon) V\quad \text{for all}\; x\in K^c,
\end{equation}
for some compact ball $K$ and $\varepsilon>0$.
Then $\lambda^-_1(F)$ is simple.
\end{theorem}

\subsection{Motivation}\label{S-motiv}
One of the important examples of  $F$ comes from the control theory. In particular, we may consider
\begin{equation}\label{E2.8}
F(D^2\phi,D\phi,\phi, x)\,=\,\sup_\alpha\{\trace(a_\alpha(x)D^2\phi(x))+ b_\alpha(x)\cdot D\phi(x) + c_\alpha(x)\phi(x)\}=\, \sup_\alpha \{L_\alpha\phi + c_\alpha\phi\}\,,
\end{equation}
where $\alpha$ varies over some index set $\cI$, $\lambda(x) I\leq a_\alpha(x)\leq \Lambda(x) I$,
and $\sup_{\alpha\in\cI}|b_\alpha(x)|, \sup_{\alpha\in\cI}|c_\alpha(x)|$ are locally bounded. The eigenvalue
problem corresponding to the operator $F$ appears in the study of \textit{risk-sensitive controls}.
See for instance, \cite{ABBS, ABS19} and references therein. To elaborate, suppose that $\cI$ is a
compact subset subset of $\RR^m$. Let $\Uadm$ be the collection of
Borel measurable maps $\alpha:\RN\to\cI$. Note that constant functions
are also included in $\Uadm$. This set $\Uadm$ represents the collection of all Markov controls. Given $\alpha\in\Uadm$,
suppose that $X_\alpha$ is the Markov
diffusion process with generator $L_\alpha$. Denote the law of $X_\alpha$ by $\mathbb{P}_\alpha$
and $\Exp_\alpha[\cdot]$ is the expectation operator associated with it. Consider the maximization problem
$$\Lambda\,=\, \sup_{\alpha\in\Uadm}\, \limsup_{T\to\infty}\, 
\frac{1}{T}\log\Exp_\alpha\left[e^{\int_0^T c_\alpha(X_t) \D{t}}\right].$$
Then under reasonable hypothesis, one can show that $\Lambda$ is an eigenvalue of $F$ (i.e. $\Lambda\in\mathcal{E}^+$)
and for many practical reasons it is desirable that $\Lambda=\lambda^+_1(F)$. Also, simplicity
of $\lambda^+_1(F)$ is important to find an optimal strategy or control. We refer the readers to 
\cite{ABBS, ABS19} for more details on this problem.

\section{Proofs of main results}\label{S-proof}
In this section we prove \cref{T2.1,T2.2,T2.3,T2.4,T2.5,T2.6,T2.7}. Let us start by recalling the following
Harnack inequality from \cite[Theorem~3.6]{QS08} which will be crucial for our proofs. The
result in \cite[Theorem~3.6]{QS08} is stated for $L^N$-viscosity solutions and also applies to
$L^N$-strong solutions due to \cite[Lemma~2.5]{CCKS}.
\begin{theorem}\label{T3.1}
Let $\Omega\subset\RN$ be bounded. Let $u\in\cC(\bar\Omega)\cap\Sobl^{2,N}(\Omega)$ and $f\in L^N(\Omega)$ satisfy $u\geq 0$ in $\Omega$ and
\begin{align*}
\cM^+_{\lambda,\gamma}(x, D^2 u) + \gamma |Du| +\delta u & \geq \, f \quad \text{in}\; \Omega,
\\
\cM^-_{\lambda,\gamma}(x, D^2 u) - \gamma |Du| -\delta u & \leq \, f \quad \text{in}\; \Omega.
\end{align*}
Then for any compact set $K\Subset\Omega$ we have
$$\sup_{K} u \;\leq\; C\,[\inf_{K} u + \norm{f}_{L^N(\Omega)}]\,,$$
for some constant $C$ dependent on $K, \Omega, N, \gamma, \delta$, $\min_{\Omega}\lambda$ and
$\max_{\Omega}\Lambda$.
\end{theorem}
Next we prove \cref{T2.1}. The idea is the following: we show using the Harnack inequality and stability
estimate that the Dirichlet principal eigenpair in $\sB_n$ converges to a principal eigenpair in $\RN$. For
any $\lambda<\lambda^+_1(F)$ or $\lambda<\lambda^{-}_1(F)$ we use the refined maximum principle in bounded domains
and then stability estimate to pass the limit.
We spilt the proof of \cref{T2.1} in \cref{L3.1} and \cref{L3.2}.
\begin{lemma}\label{L3.1}
It holds that $\mathcal{E}^+=(-\infty, \lambda^+_1(F)]$.
\end{lemma}

\begin{proof}
Let $\lambda^+_1(F, \sB_n)$ be the Dirichlet principal eigenvalue in $\sB_n$
corresponding to the positive principal eigenfunction. Existence of $\lambda^+_1(F, \sB_n)$ follows from 
\cite[Theorem~1.1]{QS08}. For notational economy we denote $\lambda^+_1(F, \sB_n)=\lambda^+_{1,n}$
and $\lambda^+_1(F)=\lambda^+_1$. We also set $E_p(\Omega)=\Sobl^{2,p}(\Omega)\cap\cC(\bar\Omega)$.
We divide the proof into two steps.

\noindent{\bf Step 1.}\, We show that $\lim_{n\to\infty} \lambda^+_{1,n}=\lambda^+_1$ and $\lambda^+_1\in\mathcal{E}^+$. It is obvious from the definition that $\lambda^+_{1,n}$ is decreasing in $n$ and bounded below by
$\lambda^+_{1}$. Thus if $\lim_{n\to\infty} \lambda^+_{1,n}=-\infty$, we also have $\lambda^+_{1}=-\infty$
and there is nothing to prove. So we assume $\lim_{n\to\infty} \lambda^+_{1,n}\df\tilde\lambda>-\infty$.
It is then obvious that $\tilde\lambda\geq \lambda^+_{1}$. 
From \cite[Theorem 1.1]{QS08} we have $\psi_{1,n}^+\in E_p(\sB_n),\,  \forall\, p<\infty $,
such that $\psi^+_{1,n}>0$ in $\sB_n$, $\psi^+_{1,n}=0$ on $\partial \sB_n$ and satisfies
\begin{equation}\label{EL3.1A}
F(D^2\psi^+_{1,n},D\psi_{1,n}^+,\psi_{1,n}^+,x)\,=\,-\lambda_{1,n}^+\psi_{1,n}^+\quad \text{in}\; \sB_n\,,
\end{equation}
for all $n\geq 1$.
Normalize each $\psi_{1,n}^+$ by choosing $\psi_{1,n}^+(0)=1$. Fix any compact $K\subset\RN$ such that $0\in K$
and choose $n_0$ large so that $K\Subset \sB_m$ for all $m\geq n_0$. Applying \cref{T3.1} on \eqref{EL3.1A}
we find
a constant $C=C(n_0)$ satisfying
\begin{align*}
\sup_{K}\,\psi_{1,n}^+\,\leq\, C\,\inf_{K}\psi_{1,n}^+\,\leq C\,\psi_{1,n}^+(0)\,=\,C.
\end{align*}
Thus applying \cite[Theorem 3.3]{QS08} we obtain, for $p>N$, that 
\begin{align*}
\norm{\psi_{1,n}^+}_{\Sob^{2,p}(K)}\,\leq\, C \quad \forall\, n>n_0\,.
\end{align*}
Since $K$ is arbitrary, using a standard diagonalization argument we can find a non-negative
$\varphi^+\in E_p(\RN),\, \forall\, p<\infty$, such that $\psi_{1,n}^+\to\varphi^+$ in $\Sobl^{2, p}(\RN)$,
upto a subsequence. Hence by \cite[Theorem~3.8 and Corollary~3.7]{CCKS} we obtain
\begin{align*}
F(D^2\varphi^+,D\varphi^+,\varphi^+,x)=-\tilde\lambda\varphi^+\quad \text{in}\; \RN, 
\quad  \varphi^+(0)=1.
\end{align*}
Again, applying \cref{T3.1} we have $\varphi^+>0$. Thus, $\tilde\lambda\leq \lambda^+_1$. This shows
$\tilde\lambda= \lambda^+_1$ and $\lambda^+_1\in\mathcal{E}^+$.

\noindent{\bf Step 2.}\, We show that $\mathcal{E}^+=(-\infty, \lambda^+_1]$. It is obvious that 
$\mathcal{E}^+\subset(-\infty, \lambda^+_1]$. To show the reverse relation we consider $\lambda<\lambda^+_1$.
We choose a sequence $\{f_n\}_{n\geq 1}$ of continuous, non positive, non-zero functions
satisfying 
\begin{align*}
\supp(f_n)\subset \sB_n\setminus \overline{\sB}_{n-1}\quad \text{ for all }\; n\in\mathbb{N}.
\end{align*}
Denote by $\tilde F = F + \lambda$.  Then 
$\lambda^+_1(\tilde{F}, \sB_n)=\lambda^+_{1,n}-\lambda\geq \lambda^+_1-\lambda>0$. 
Therefore, by \cite[Theorem~1.5 and Theorem 1.8]{QS08}, there exists a unique 
non-negative $u^n\in E_p(B_n)$, for all $p\geq N$, which satisfies 
\begin{align}\label{EL3.1B}
\tilde{F}(D^2 u^n,Du^n,u^n,x)\,=\,f_n\quad \text{in}\; \sB_n,\quad \text{and}
\quad u^n=0\quad \text{ on }\;\partial \sB_n\,.
\end{align}
By the strong maximum principle \cite[Lemma~3.1]{QS08} it follows that $u^n>0$ in $\sB_n$.
For natural number $n\geq 2$ we define
\begin{align*}
v^n(x)\df\frac{u^n(x)}{u^n(0)}\,.
\end{align*}
Clearly, $v^n\in E_p(\sB_{n-1}),\,\forall\, p<\infty$, positive in $\sB_{n-1}$ and $v^n(0)=1$.
Also, by \eqref{EL3.1B},
\begin{align*}
F(D^2v^n,Dv^n,v^n,x)=-\lambda v^n\quad \text{in}\; \sB_{n-1}\,.
\end{align*}
Now we continue as in Step 1 and extract a subsequence of $v^n$ that converges in $\Sobl^{2, p}(\RN)$
 to some positive $\varphi\in E_p(\RN), \, \forall\, p<\infty$,  and satisfies 
$$F(D^2\varphi,D\varphi,\varphi,x)=-\lambda \varphi\quad \text{in}\; \RN\,.$$
This gives us $\lambda\in \mathcal{E}_+$. Thus $\mathcal{E}_+=(-\infty, \lambda^+_1]$.
\end{proof}

Next lemma concerns the eigenvalues with negative eigenfunctions.

\begin{lemma}\label{L3.2}
It holds that $\mathcal{E}^-=(-\infty, \lambda^-_1(F)]$.
\end{lemma}

\begin{proof}
Idea of the proof is similar to \cref{L3.1}. Let $\lambda^-_1(F, \sB_n)$ be the Dirichlet principal eigenvalue in $\sB_n$ corresponding to the negative principal eigenfunction \cite[Theorem~1.1]{QS08}.
For simplicity we denote $$\lambda^-_1(F, \sB_n)=\lambda^-_{1,n}
\text{ and } \lambda^-_1(F)=\lambda^-_1.$$
We divide the proof of into two steps.

\noindent{\bf Step 1.}\, We show that $\lim_{n\to\infty} \lambda^-_{1,n}=\lambda^-_1$ and $\lambda^-_1\in\mathcal{E}^-$. It is obvious from the definition that $\lambda^-_{1,n}$ in decreasing in $n$ and bounded below by
$\lambda^-_{1}$. Thus if $\lim_{n\to\infty} \lambda^-_{1,n}=-\infty$, we also have $\lambda^-_{1}=-\infty$
and there is nothing to prove. So we assume $\lim_{n\to\infty} \lambda^-_{1,n}\df\widehat\lambda>-\infty$.
It is then obvious that $\widehat\lambda\geq \lambda^-_{1}$.
From \cite[Theorem 1.1]{QS08}, for all $n\in \mathbb{N}$, we have $\psi_{1,n}^-\in E_p(\sB_n),\, \forall\, p<\infty$, such that $\psi_{1,n}^-<0$ in $\sB_n$, $\psi_{1,n}^-=0$ in $\partial \sB_n$, and  
\begin{align}\label{EL3.2A}
F(D^2\psi_{1,n}^-,D\psi_{1,n}^-,\psi_{1,n}^-,x)\,=\,-\lambda_{1,n}^-\psi_{1,n}^-\quad  \text{ in}\; \sB_n\,.
\end{align}
Normalize each $\psi_{1,n}^-$ by fixing $\psi_{1,n}^-(0)=-1$. Denoting $G(M, p, u, x)=-F(-M, -p, -u, x)$
we find from \eqref{EL3.2A}
\begin{align*}
G(D^2\phi_{1,n}^-,D\phi_{1,n}^-,\phi_{1,n}^-,x)\,=\,-\lambda_{1,n}^-\phi_{1,n}^-\quad  \text{ in}\; \sB_n\,,
\end{align*}
for $\phi^-_{1,n}=-\psi^-_{1, n}\geq 0$. Since $G$ satisfies conditions (H1), (H3) and (H4), 
\cref{T3.1} applies. Then using \eqref{EL3.2A} and \cite[Theorem~3.3]{QS08}, we can obtain
locally uniform $\Sobl^{2, p}$ bounds on $\phi^-_{1,n}$.
Now apply the arguments of Step 1 in the proof of \cref{L3.1} to show that $\lim_{n\to\infty} \lambda^-_{1,n}=\lambda^-_1$ and $\lambda^+_1\in\mathcal{E}^-$.

\noindent{\bf Step 2.} As discussed in \cref{L3.1}, it is enough to show that for any $\lambda<\lambda^{-}_1$
we have $\lambda\in\mathcal{E}^-$. Consider a sequence $\{f_n\}_{n\geq 1}$ of continuous, non negative, non-zero functions satisfying 
\begin{align*}
\supp(f_n)\subset \sB_n\setminus \overline{\sB}_{n-1}\quad \text{ for all }\; n\in\mathbb{N}.
\end{align*}
Denote by $\tilde F = F + \lambda$. Then 
$\lambda^-_1(\tilde{F}, \sB_n)=\lambda^-_{1,n}-\lambda\geq \lambda^-_1-\lambda>0$.
Therefore, by \cite[Theorem 1.9]{QS08}, there exists a non-zero, non positive $u^n\in E_p(\sB_n)$,
 for all $p\geq N$, satisfying 
\begin{align*}
\tilde{F}(D^2u^n,Du^n,u^n,x)=f_n\quad \text{in}\; \sB_n,\quad  \text{ and }\quad  u^n=0
\quad  \text{ in }\; \partial \sB_n.
\end{align*}
Since $G$ satisfies (H3) we can apply the strong maximum principle \cite[Lemma~3.1]{QS08} to obtain that 
$u^n<0$ in $\sB_n$. Now repeat the arguments of Step 2 in the proof of \cref{L3.1} to conclude that $\lambda\in\mathcal{E}^-$.
This completes the proof.
\end{proof}

\begin{proof}[Proof of \cref{T2.1}]
The proof follows from \cref{L3.1,L3.2}.
\end{proof}

The following (standard) existence result will be required. 
\begin{lemma}\label{L3.3}
Suppose that $\underline{u},\bar{u}\in E_p(\Omega)$, for some $p\geq N$ and $\Omega$ is a smooth bounded domain, and
$\bar{u}$ ($\underline{u}$) is a supersolution(subsolution) of $F(D^2u, Du, u, x)=f(x,u)$ in $\Omega$ for
some $f\in L^\infty_{\rm loc}(\bar{\Omega}\times\RR)$. Assume that $f$ is locally Lipschitz in its second argument uniformly (almost surely)
with respect to the first argument and $\underline{u}\leq 0, \bar{u}\geq 0$ on $\partial\Omega$.
Then there exists $u\in E_p(\Omega)$ with $\underline{u}\leq u\leq \bar{u}$
in $\Omega$ and satisfies
\[ \begin{split} 
F(D^2u,Du,u,x)&=f(x,u) \quad \text{ in }\Omega,
\\
v &=0\qquad \text{ on }\partial\Omega\,.
\end{split}
\]
\end{lemma}

\begin{proof}
The proof is based on monotone iteration method. See also \cite[Lemma~4.3]{QS08} for a similar argument.
Define the operator $\tilde{F}=F-\theta$ in such a way that $\lambda_1^+(\tilde{F},\Omega)>0$. We may 
choose $\theta$
large enough so that
$$\theta > \mathrm{Lip}(f(x, \cdot)\, \text{on}\; [\inf_{\Omega}\underline{u}, \sup_{\Omega}\bar{u}])
\quad \text{almost surely for}\; x\in \Omega,$$ 
and $\tilde{F}$ is proper i.e., decreasing in $u$.
 Also, note that $\tilde{F}$ satisfying (H1)-(H4). Now we define the monotone sequence.
Denote by $v_0=\underline{u}$, and for each $n\geq 0$, we define
		\[ \begin{cases} 
		\tilde{F}(D^2v_{n+1},Dv_{n+1},v_{n+1},x)=f(x,v_{n})-\theta v_{n} & \text{ in }\Omega\,,
		\\
		v_{n+1}=0 & \text{ on }\partial\Omega\,.
		\end{cases}
		\]
Existence of $v_{n+1}\in E_p$ follows from \cite[Theorem~3.4]{QS08}. Also, since $\tilde{F}$ is proper,
we can apply comparison principle \cite[Theorem~3.2]{QS08} to obtain $v_0\leq v_1\leq v_2\leq\cdots\leq \bar{u}$. It is then standard to show that $v_n\to u$ in $\cC(\bar{\Omega})$ for some $u\in E_p(\Omega)$ and 
$u$ is our required solution (see for instance, \cite[Lemma~4.3]{QS08} ). This completes the proof.
\end{proof}

Applying \cref{L3.3} we obtain the following.
\begin{theorem}\label{T3.2}
It holds that $\plamplus(F)\leq\lambda_1^+(F)$ and $\plamminus(F)\leq\lambda_1^-(F)$.
\end{theorem}

\begin{proof}
We divide the proof into two steps.

\noindent{\bf Step 1.} We show that $\plamplus(F)\leq\lambda_1^+(F)$. Replacing
$F$ by $F-\lambda_1^+(F)$ we may assume that $ \lambda_1^+(F)=0$.
Considering any $\lambda$ satisfying $\lambda>0$ we show that $\plamplus(F)\leq\lambda$. 
Recall from \cref{L3.1} that $\lambda^{+}_{1}(F, \sB_n)\searrow \lambda_1^+(F)$ as $n\to\infty$. Thus
we can find $k$ large enough satisfying $\lambda>\lambda^{+}_{1}(F, \sB_k)>\lambda_1^+(F)=0$.
Let $\psi_{k}^+\in E_p(\sB_k), p<\infty,$ satisfy
\[\begin{split}
F(D^2\psi_{k}^+,D\psi_{k}^+,\psi_{k}^+,x)&=\,-\lambda_{1,k}^+\psi_{k}^+\quad \text{in}\; \sB_k,
\\
\psi_{1}^+>0\; \text{in}\; \sB_k,\quad \psi_{k}^+ &=\,0\; \text{in}\; \partial \sB_k,
\end{split}
\]
where $\lambda^{+}_{1}(F, \sB_k)=\lambda_{1,k}^+$. 
Let $\tilde{\delta}=\sup_{\sB_k}\delta$ where $\delta$ is given by (H3).
Normalize $\psi_{k}^+$ so that 
$$
\norm{\psi_{k}^+}_{L^\infty(\sB_k)}=\min\left\{1, \frac{\lambda-\lambda_{1,k}^+}{\lambda+\tilde{\delta}}\right\}\,.
$$
Now we plan to find a bounded, positive solution of 
\begin{align}\label{ET3.2A}
F(D^2u,Du,u,x)=(\lambda+ c^+(x))u^2-\lambda u \quad \text{in}\; \RN\,,
\end{align}
where $c(x)=F(0,0,1,x)\in L^\infty_{\rm loc}(\RN)$.
This would imply $F(D^2u,Du,u,x)\geq -\lambda u$, and therefore, $\plamplus(F)\leq\lambda$. Thus to complete
the proof of Step 1 we only need to establish \eqref{ET3.2A}.

Let $\bar{u}=1$ and $\underline{u}=\psi^+_{k}$. Note that $\bar{u}$ is a supersolution in $\RN$ and
$\underline{u}$ is a subsolution in $\sB_k$.
Now fix any ball $\sB$ containing $\sB_k$. Since $0$ is a
subsolution, by \cref{L3.3}, we find $v\in E_p(\sB), p<\infty$, with $0\leq v\leq 1$ and satisfies
$$
F(D^2v,Dv,v,x)=(\lambda+\tilde{\delta})v^2-\lambda v \quad \text{in}\; \sB,\quad v=0\; \text{on}\; 
\partial\sB\,.
$$
The proof of \cref{L3.3} also reveals that $v\geq \psi^+_k$ in $\sB_k$. Now choosing a sequence of 
$\sB$ increasing to $\RN$, and the interior estimate \cite[Theorem~3.3]{QS08} we can find a
subsequence locally converging to a solution $u$ of \eqref{ET3.2A}. Positivity of $u$ follows
from \cref{T3.1}.

\noindent{\bf Step 2.} We next show that $\plamminus(F)\leq\lambda_1^-(F)$.
Replacing $F$ by $F-\lambda_1^-(F)$ we may assume that $ \lambda_1^-(F)=0$.
Considering any $\lambda$ satisfying $\lambda>0$ we show that $\plamminus(F)\leq\lambda$.
As done in Step 1, we can choose $k$ large enough so that $\lambda>\lambda^{-}_1(F, \sB_k)\df\lambda_{1,k}^-$
and there exists $\psi^{-}_k\in E_p(\sB_k)$ satisfying
\[\begin{split}
F(D^2\psi_{k}^-,D\psi_{k}^-,\psi_{k}^-,x)&=\,-\lambda_{1,k}^-\psi_{k}^+\quad \text{in}\; \sB_k,
\\
\psi_{k}^-<0\; \text{in}\; \sB_k,\quad \psi_{k}^- &=\,0\; \text{in}\; \partial \sB_k.
\end{split}
\]
Normalize $\psi_{k}^-$ so that 
$$\norm{\psi_{k}^-}_{L^\infty(\sB_k)}=\min\left\{1, \frac{\lambda-\lambda_{1,k}^-}{\lambda+\tilde{\delta}}\right\},
$$
where $\tilde\delta$ is same as in Step 1. Then 
$$F(D^2\psi_{k}^-,D\psi_{k}^-,\psi_{k}^-,x)\leq -(\lambda+ c^-(x))(\psi^-_k)^2-\lambda \psi^-_k
\quad \text{in}\; \sB_k.$$
Thus, using \cref{L3.3} and the arguments of Step 1, we obtain a negative, bounded solution 
$u\in\Sobl^{2,p}(\RN), p<\infty$, to
\begin{align*}
F(D^2u,Du,u,x)=-(\lambda+ c^-(x))u^2-\lambda u\leq -\lambda u.
\end{align*}
This of course, implies $\plamminus(F)\leq \lambda$. Hence the theorem.
\end{proof}

\cref{T2.2}(ii) will be proved using \cref{T2.4}. Thus we prove \cref{T2.4} first.
\begin{theorem}\label{T3.3}
Suppose that either \eqref{E2.3} and \eqref{E2.5} or \eqref{E2.4} and \eqref{E2.6} hold. 
Then $F$ satisfies $\beta^+$-MP in $\RN$ provided $\lambda_{\beta}^{\prime\prime,+}(F)>0$.
\end{theorem}

\begin{proof}
Let $u\in\Sobl^{2, N}(\RN)$ be a  function satisfying 
$$
F(D^2u,Du,u,x)\,\geq\, 0\quad  \text{in}\; \RN,\quad \text{ and } 
\quad \sup_{\RN}\ \frac{u}{\beta}\,<\,\infty\,.
$$
Also, since $\lambda_{\beta}^{\prime\prime,+}(F)>0$, there exists $\lambda>0$ 
and $\psi\in\Sobl^{2, N}(\RN)$ with the property that $\psi\geq\beta$ and
$$ F(D^2\psi,D\psi,\psi,x) + \lambda\psi\,\leq\, 0\quad  \text{in }\; \RN\,.$$ 
Multiplying $\psi$ with a suitable constant we may assume that $\psi\geq u$. 

For this proof we follow the idea of \cite[Theorem~4.2]{BR15}.
Choose a smooth positive function $\chi:\RN\rightarrow\mathbb{R}$ such that, for $|x|>1$,
\[
\chi(x)=\left\{
\begin{array}{lll}
|x|^\sigma & \text{if $\beta$ satisfies \eqref{E2.3}}\,,
\\
\exp(\sigma|x|) & \text{if $\beta$ satisfies \eqref{E2.4}}\,.
\end{array}
\right.
\]
Using (H3) and an easy computation we obtain for $x\in\sB^c_1$
\[
F(D^2\chi,D\chi,\chi,x)\leq\left\{
\begin{array}{lll}
\left[(\sigma^2+N\sigma-2\sigma)\frac{\Lambda(x)}{|x|^2}+\sigma\frac{\gamma(x)}{|x|}+\delta(x)\right]\chi
& \text{if $\beta$ satisfies \eqref{E2.3}}\,,
\\
\left[\sigma\bigg(\sigma+\frac{N-1}{|x|}\bigg)\Lambda(x)+\sigma\gamma(x)+\delta(x)\right]\chi
& \text{if $\beta$ satisfies \eqref{E2.4}}\,.
\end{array}
\right.
\]
Hence for both the cases, using \eqref{E2.5} and \eqref{E2.6} accordingly on $\overline{\sB}_1^c$, there exists a positive constant $C$ such that 
\begin{align}\label{ET3.3A}
F(D^2\chi,D\chi,\chi,x)\,\leq\, C\chi.
\end{align}
Modifying $C$, if required, we can assume \eqref{ET3.3A} to hold in $\RN$.
Now set $\psi_n=\psi + \frac{1}{n}\chi$ and define $\kappa_n=\sup_{\RN}\frac{u}{\psi_n}$. If $\kappa_n\leq 0$
then there is nothing to prove. Thus we assume $\kappa_n>0$ to reach a contradiction. Since $\psi\geq u$ it
follows that $\kappa_n\leq 1$ and $\kappa_n\leq\kappa_{n+1}$ for all $n\geq 1$. Moreover, by \eqref{E2.3} and \eqref{E2.4},
$$
\limsup_{|x|\rightarrow\infty}\,\frac{u(x)}{\psi_n(x)}\,\leq\, n\ \sup_{\RN}\frac{u}{\beta}\ 
\limsup_{|x|\rightarrow\infty}\frac{\beta(x)}{\chi (x)}\;=\,0\,.
$$ 
Hence there exist $x_n\in\RN$ such that $\kappa_n\,=\,\frac{u(x_n)}{\psi_n(x_n)}$.

Let us now estimate the term $\frac{\chi(x_n)}{n}$. Note that
$$
\frac{1}{\kappa_{2n}}\,\leq\,\frac{\psi_{2n}(x_n)}{u(x_n)}\,=\,\frac{1}{\kappa_n}-\frac{\chi(x_n)}{2n\,u(x_n)}\,,
$$
which implies 
\begin{equation*}
\frac{\chi(x_n)}{n}\leq 2\bigg(\frac{1}{\kappa_{n}}-\frac{1}{\kappa_{2n}}\bigg)u(x_n)\leq2\bigg(\frac{1}{\kappa_{n}}-\frac{1}{\kappa_{2n}}\bigg)\psi(x_n)\,.
\end{equation*}
Hence for each natural number $n$ there exist a small positive $\eta_n$ such that 
\begin{equation}\label{ET3.3B}
\frac{\chi(x)}{n}\leq \bigg(\frac{1}{\kappa_{n}}-\frac{1}{\kappa_{2n}}\bigg)\psi(x)
\quad \text{in}\; \sB_{\eta_n}(x_n).
\end{equation}
On the other hand, using convexity of $F$ with \eqref{ET3.3A} and \eqref{ET3.3B}, we get
\begin{align*}
F(D^2\psi_n,D\psi_n,\psi_n,x) & \leq  F(D^2\psi,D\psi,\psi,x)+\frac{1}{n}F(D^2\chi,D\chi,\chi,x)
\\
&\leq \left[-\lambda +  C \bigg(\frac{1}{\kappa_{n}}-\frac{1}{\kappa_{2n}}\bigg)\right]\psi(x)\,,
\end{align*}
in $\sB_{\eta_n}(x_n)$. Since $\{\kappa_{n}\}$ is a convergent sequence,
we can choose $m$ large enough so that 
\begin{align}\label{ET3.3C}
F(D^2\psi_m,D\psi_m,\psi_m,x)\,<\,0\quad \text{ in}\; \sB_{\eta_m}(x_m)\,.
\end{align}
Now note that $w=\kappa_m\psi_m-u$ is non-negative and by (H3), there exist positive $a, b$ such that
in $\sB_{\eta_m}(x_m)$ we have
\begin{align*}
\mathcal{M}_{\lambda,\Lambda}^-(x, D^2w)-a|Dw|-b w \,\leq\, \kappa_m F(D^2\psi_m,D\psi_m,\psi_m,x)-F(D^2u,Du,u,x)\,<\,0\,.
\end{align*}
By the strong maximum principle \cite[Lemma~3.1]{QS08} we then obtain $w\equiv 0$ in $\sB_{\eta_m}(x_m)$.
But this contradicts \eqref{ET3.3C} as
$$0\leq F(D^2u,Du,u,x)\,=\,\kappa_m F(D^2\psi_m,D\psi_m,\psi_m,x)<0\quad \text{ in}\; \sB_{\eta_m}(x_m)\,.$$
Therefore, $\kappa_n\leq 0$ for large $n$ and hence $u\leq 0$.
\end{proof}

In the same spirit of \cref{T3.3} we can also prove $\beta^-$-MP.
\begin{theorem}\label{T3.4}
Suppose that either \eqref{E2.3} and \eqref{E2.5} or \eqref{E2.4} and \eqref{E2.6} hold for the function $\beta$. 
Then $F$ satisfies $(-\beta)^-$-MP in $\RN$ provided $\lambda_{\beta}^{\prime\prime,-}(F)>0$.
\end{theorem}

\begin{proof}
As done in \cref{T3.3}, we choose $\lambda\in (0, \lambda_{\beta}^{\prime\prime,-}(F))$ and 
$\psi\in \Sobl^{2, N}(\RN)$ satisfying $\psi\leq-\beta$ and
$$ F(D^2\psi,D\psi,\psi,x) + \lambda\psi\,\leq\, 0\quad  \text{in }\; \RN\,.$$
Let $u\in\Sobl^{2, N}(\RN)$ be a function satisfying 
$$
F(D^2u,Du,u,x)\,\leq\, 0\quad  \text{in}\; \RN,\quad \text{ and } 
\quad \sup_{\RN}\ \frac{u}{(-\beta)}\,<\,\infty\,.
$$
We need to show that $u\geq 0$.
To the contrary, we suppose that $u$ is negative somewhere in $\RN$.
Multiplying $\psi$ with a suitable positive constant we may assume $\psi\leq u$. Consider
the function $\chi$ from \cref{T3.3} and note that \eqref{ET3.3A} holds. 
Set $\psi_n(x)=\psi(x)-\frac{1}{n}\chi(x)$ and $\kappa_n\df\sup_{\RN} \frac{u}{\psi_n}$. It can be easily 
checked that $(\kappa_n)_{n\in\mathbb{N}}$ is positive, increasing and bounded by $1$. Furthermore,
$\kappa_n=\frac{u(x_n)}{\psi_n(x_n)}$ for some $x_n\in\RN$. Then repeating a similar calculation we
find that
for each natural number $n$ there exist a small positive $\eta_n$ satisfying
$$
-\frac{\chi(x)}{n}\geq \bigg(\frac{1}{\kappa_{n}}-\frac{1}{\kappa_{2n}}\bigg)\psi(x)
\quad \text{in}\; \sB_{\eta_n}(x_n)\,.
$$
Then using convexity, \eqref{ET3.3A} and above estimate,  we obtain 
\begin{align*}
F(D^2\psi_n,D\psi_n,\psi_n,x) 
& \geq  F(D^2\psi,D\psi,\psi,x)-\frac{1}{n}F(D^2\chi,D\chi,\chi,x) 
\\ 
&\geq   \left[-\lambda \psi(x)- C\frac{\chi(x)}{n}\right]
\\
&\geq  \left[-\lambda +  C \big(\frac{1}{\kappa_{n}}-\frac{1}{\kappa_{2n}}\big)\right]\psi(x)\,, 
\end{align*}
in $\sB_{\eta_n}(x_n)$. As $\psi(x)$ is negative and $\{\kappa_{n}\}$ is convergent, we can choose $m$ large enough such that 
\begin{align}\label{ET3.4A}
F(D^2\psi_m,D\psi_m,\psi_m,x)\,>\,0\quad \text{in}\; \sB_{\eta_m}(x_m).
\end{align}
Note that $w\df\kappa_m\psi_n-u$ is a non-positive function vanishing at $x_m$. Repeating the
arguments of \cref{T3.3} we find positive constants $a_1, b_1$ satisfying
$$\mathcal{M}_{\lambda,\Lambda}^+(x, D^2 w)+ a_1|Dw|-b_1w\,\geq\, 0\,,$$
in $\sB_{\eta_m}(x_m)$. This of course, implies $w\equiv 0$ in $\sB_{\eta_m}(x_m)$ which is a contradiction
to \eqref{ET3.4A}. Thus it must hold that $u\geq 0$.
\end{proof}

\begin{proof}[Proof of \cref{T2.4}]
The proof follows by combining \cref{T3.3,T3.4}.
\end{proof}

Now we prove \cref{T2.5}.
\begin{proof}[Proof of \cref{T2.5}]
First we consider (i). To the contrary, suppose that $\lambda_{\beta}^{\prime,+}(F)<0$.
Then there exists $\lambda<0$ such that $\lambda_{\beta}^{\prime, +}(F)<\lambda<0$ and there exists
$\psi\in\Sobl^{2, N}(\RN)$ satisfying
$$
0<\psi\leq \beta,\quad  F(D^2\psi,D\psi,\psi,x) + \lambda\psi\,\geq\,0\,.
$$
This of course, implies $F(D^2\psi,D\psi,\psi,x)\geq-\lambda\psi\,>\, 0$
and $\sup\frac{\psi}{\beta}\leq 1$. This clearly violates $\beta^+$-MP.

Next we consider (ii). Again, we suppose that $\lambda_{\beta}^{\prime,-}(F)<0$. 
Then there exists $\lambda<0$ such that $\lambda_{\beta}^{\prime,-}(F)<\lambda<0$ and
 there exists $\psi\in\Sobl^{2, N}(\RN)$ satisfying 
$$
0>\psi\geq -\beta,\quad  F(D^2\psi,D\psi,\psi,x)+\lambda\psi\,\leq\,0\,.$$
This gives $F(D^2\psi,D\psi,\psi,x)\leq-\lambda\psi< 0$ and $\sup\frac{\psi}{(-\beta)}\leq 1$. This clearly violates $(-\beta)^-$-MP.
\end{proof}

Now we can prove \cref{T2.2}(ii).
\begin{theorem}\label{T3.5}
Assume that either \eqref{E2.5} or \eqref{E2.6} holds. Then we have 
$$\pplamplus(F)\,\leq\, \plamplus(F), \quad \text{and}\quad \pplamminus(F)\,\leq\,\plamminus(F)\,.$$
\end{theorem}

\begin{proof}
Let us first show that $\pplamplus(F)\,\leq\, \plamplus(F)$. To the contrary,
suppose that there exists $\lambda$ with $\lambda<\pplamplus(F)$ and $\plamplus(F)<\lambda$. Then there
exists  positive $\psi\in\Sobl^{2,N}(\RN)\cap L^\infty(\RN)$ such that $F(D^2\psi,D\psi,\psi,x)+\lambda\psi\geq\,0$. Also, note that $\pplamplus(F+\lambda)=\pplamplus-\lambda>0$. 
By \cref{T3.3}, the operator $F+\lambda$ satisfies $+$MP. Therefore, $\psi\leq 0$ which contradicts the
fact $\psi>0$. Hence we must have $\pplamplus(F)\,\leq\, \plamplus(F)$.

We prove the second claim. To the contrary,
suppose that there exists $\lambda$ with $\lambda<\pplamminus(F)$ and $\plamminus(F)<\lambda$. 
Then there exists a negative function $\psi\in\Sobl^{2,N}(\RN)\cap L^\infty(\RN)$ such that  $F(D^2\psi,D\psi,\psi,x)+\lambda\psi\leq\,0$. Also, we have $\pplamminus(F+\lambda)=\pplamminus(F)-\lambda>0$, and therefore,
the operator $F+\lambda$ satisfies $-$MP. This gives $\psi\geq 0$ which contradicts the fact $\psi<0$.
Hence we must have $\pplamminus(F)\,\leq\,\plamminus(F)$.
\end{proof}

\begin{proof}[Proof of \cref{T2.2}]
The proof follows by combining \cref{T3.2,T3.5}.
\end{proof}

Our next result should be compared with \cite[Theorem~7.6]{BR15}. Recall that for a smooth domain
$\Omega$
\begin{align*}
\pplamplus(F,\Omega) &=\sup\{\lambda\;:\;\exists\;\psi\in \Sobl^{2, N}(\Omega),\; \inf_{\Omega}\psi>0\; 
\text{and}\;	F(D^2\psi,D\psi,\psi,x) + \lambda\psi\leq 0\; \text{ in }\; \Omega\}\,,
\\
\pplamminus(F,\Omega)&=\sup\{\lambda\;:\;\exists\;\psi\in \Sobl^{2, N}(\Omega),\; \sup_{\Omega}\psi<0\; 
\text{and}\;	F(D^2\psi,D\psi,\psi,x) + \lambda\psi\geq 0\; \text{ in }\; \Omega\}\,.
	\end{align*}
\begin{theorem}\label{T3.6}
It holds that
$$
\pplamplus(F)\,=\,\min\{\lambda_{1}^{+}(F),\lim_{r\rightarrow\infty}\pplamplus(F,
\bar{\sB}_r^c)\}\,.
$$
\end{theorem}

\begin{proof}
Notice that the function $\pplamplus(r)\df\pplamplus(F,\bar{\sB}_r^c)$ is an increasing function with respect to $r$ and 
\begin{align*}
\pplamplus(F)\,\leq\, \lim_{r\rightarrow \infty} \pplamplus(r)\,.
\end{align*}
Also, from definition we already have $\pplamplus(F)\leq \lambda_{1}^{+}(F)$. Combining these two we obtain 
\begin{align*}
\pplamplus(F)\,\leq\,\min\{\lambda_{1}^{+}(F),\lim_{r\rightarrow\infty}\pplamplus(F,
\bar{\sB}_r^c)\}\,.
\end{align*}
Let us now show that the above inequality can not be strict. That is, for every
$$
\lambda\,<\,\min\{\lambda_{1}^{+}(F),\lim_{r\rightarrow\infty}\pplamplus(F,
\bar{\sB}_r^c)\}\,,
$$
we have $\pplamplus(F)\geq \lambda$. To do this we need to construct a positive supersolution of the operator 
$F+\lambda$ in the admissible class of $\pplamplus(F)$. 
Choose a positive number $R$ so that $\lambda<\pplamplus(R)$. Then there exists positive function
$\phi\in\Sobl^{2,N}(\overline{\sB}_R^c)$ with  $\inf_{\sB^c_R}\phi>0$ and 
$F(D^2\phi,D\phi,\phi,x) + \lambda\phi\leq 0$ in $\overline{\sB}_R^c$.
We claim that there exists a function $\varphi\in\Sobl^{2, p}(\sB^c_{R+1}), p>N,$ with $\inf_{\sB^c_{R+1}}\varphi\geq 1$
and satisfies
\begin{equation}\label{ET3.6A}
F(D^2\varphi,D\varphi,\varphi,x) + \lambda\varphi\leq 0\quad \text{in}\; \sB^c_{R+1}\,.
\end{equation}
Let us first complete the proof assuming \eqref{ET3.6A}.  By Morrey's inequality we see that 
$\varphi\in \cC^1(\bar\sB_{R+1}^c)$. Consider a positive 
eigenfunction $\psi\in \Sobl^{2,N}(\RN)$ associated to $\lambda_{1}^{+}(F)$.
Choose a non-negative function $\chi\in \cC^2(\RN)$ such that 
$\chi=0$ in $\sB_{R+2}$ and $\chi=1$ in $\sB_{R+3}^c$.
For $\epsilon>0$, define $u\df\psi+\epsilon\chi\varphi$. Using convexity of
$F$ we can write 
\begin{align*}
F(D^2u,Du,u,x)\leq F(D^2\psi,D\psi,\psi,x)+\epsilon F(D^2(\chi\varphi),D(\chi\varphi),(\chi\varphi),x)\,.
\end{align*}
From the construction we can immediately say that $F(D^2u,Du,u,x)+\lambda u\leq0$ in $\sB_{R+2}\cup \sB_{R+3}^c$. We are left with the annuals region $\bar\sB_{R+3}\setminus \sB_{R+2}$. In this compact set we have
\begin{align*}
&F(D^2u,Du,u,x)+\lambda u 
\\
& \leq\, (\lambda - \lambda_{1}^{+}(F))\psi + \epsilon\left[F(D^2(\chi\varphi),D(\chi\varphi),(\chi\varphi),x)+\lambda\chi\varphi\right]
\\
& = (\lambda - \lambda_{1}^{+}(F))\psi + \epsilon \left[F(\chi D^2\varphi+2D\chi\cdot D\varphi
+\varphi D^2\chi,\chi D\varphi+\varphi D\chi,\chi\varphi,x)+\lambda\chi\varphi\right]
\\
& \leq  (\lambda - \lambda_{1}^{+}(F))\psi + \epsilon\chi \big[F(D^2\varphi,D\varphi,\varphi,x)+\lambda\varphi\big]+\epsilon F(2D\chi\cdot D\varphi+\varphi D^2\chi,\varphi D\chi,0,x)
\\
& \leq (\lambda - \lambda_{1}^{+}(F))\psi + \epsilon C <0,
\end{align*}
for $\epsilon$ small enough, where we have again used convexity of $F$. This of course, implies
$\pplamplus(F)\geq \lambda$, as required.

To complete the proof we only need to show \eqref{ET3.6A}. To this end, we may assume that $\inf\phi\geq 2$.
Let $c(x)=F(0,0,1,x)+\lambda$ and define $f(x, u)=|c(x)|f(u)$ where $f:\RR\to(-\infty, 0]$ is a Lipschitz
function with the property that $f(1)=-1$, $f(t)=0$ for $t\geq 2$. Then $\bar{u}=\phi$ is 
supersolution to 
\begin{equation*}
F(D^2u,Du,u,x) + \lambda u= f(x, u) \quad \text{in}\; \sB^c_{R}\,,
\end{equation*}
and $\underline{u}=1$ is a subsolution. The existence of a solution to \eqref{ET3.6A} follows by constructing
solutions (squeezed between $\bar{u}$ and $\underline{u}$) in an increasing sequence of bounded domains
in $\sB^c_R$ and
the passing to the limit using local stability bound \cite[Theorem~3.3]{QS08}. To construct a solution
in any smooth bounded domain we may follow the idea of \cref{L3.3} with the help of 
general existence results from \cite[Theorem~4.6]{W09} which deals with nonzero boundary condition.
\end{proof}

Now we would like to see if a result analogous to \cref{T3.6} holds for $\pplamminus(F)$. Denote by
$$G(S, p, u, x)=-F(-M, -p, -u, x).$$ It is easily seen that $G$ is a concave operator and
$\pplamminus(F)=\pplamplus(G)$. But we can not apply the arguments of \cref{T3.6} for concave operators.
To obtain the results we impose a mild condition at {\it infinity}.
\begin{theorem}\label{T3.7}
Suppose that 
\begin{align}\label{ET3.7A}
  \lim_{r\rightarrow\infty}\pplamminus(F,\overline{\sB}_r^c )\, =\,
   \lim_{r\rightarrow\infty}\pplamminus(G,\overline{\sB}_r^c )\,.
  \end{align}
Then we have
$$\pplamminus(F)\,=\,\min\left\{\lambda_{1}^{-}(F),
\lim_{r\rightarrow\infty}\pplamminus(F,\overline{\sB}_r^c )\right\}\,.
$$
\end{theorem}

\begin{proof}
It is easy to see that
\begin{align*}
\pplamminus(F)\,\leq\,\min\left\{\lambda_{1}^{-}(F),
\lim_{r\rightarrow\infty}\pplamminus(F,\overline{\sB}_r^c )\right\}\,.
\end{align*}
As done in \cref{T3.6}, we show that the above inequality can be strict. So we consider any 
\begin{equation}\label{ET3.7B}
\lambda\,<\, \min\left\{\lambda_{1}^{-}(F),
\lim_{r\rightarrow\infty}\pplamminus(F,\overline{\sB}_r^c )\right\}\,,
\end{equation}
and show that $\pplamminus(F)\geq \lambda$. We now construct a subsolution of the operator $F+\lambda$ in the admissible class of $\pplamminus(F)$. Using \eqref{ET3.7A} and \eqref{ET3.7B} we find a positive $R$
so that
$$\lambda\,<\, \pplamminus(G,\overline{\sB}_R^c)\,.$$
Hence repeating the arguments of \cref{T3.6} we can find
$\varphi\in\Sobl^{2,p}({\sB}_{R+1}^c)$, $p>N$, with $\sup_{{\sB}_{R+1}^c}\varphi<0$
 and $G(D^2\varphi,D\varphi,\varphi,x) + \lambda\varphi\geq 0$ in $\sB_{R+1}^c$.
 By Morrey's inequality $\varphi\in \cC^1(B_{R+1}^c)$. Also, consider a negative eigenfunction 
 $\psi\in\Sobl^{2,N}(\RN)$ associated to $\lambda_{1}^{-}(F)$. Let $\chi$ be the cut-off function in
 \cref{T3.6} and define $u=\psi+\epsilon\chi\varphi$ for $\epsilon>0$. Since, by convexity,
\begin{align*}
F(D^2u,Du,u,x)\,\geq\, F(D^2\psi,D\psi,\psi,x) + \epsilon G(D^2(\chi\phi),D(\chi\phi),(\chi\phi),x),
\end{align*}
repeating a calculation analogous to \cref{T3.6} we find that for some $\epsilon$ small
$F(D^2u,Du,u,x) + \lambda u\geq 0$ in $\RN$. Thus we get $\pplamminus(F)\geq \lambda$.
\end{proof}

To this end, we define $c(x)=F(0,0,1,x)$ and $d(x)=F(0,0,-1, x)$. Our next result is a
generalization to \cite[Proposition~1.11]{BR15}.

\begin{proposition}\label{P3.1}
Define $\zeta=\limsup_{|x|\rightarrow\infty}c(x)$ and $\xi= \limsup_{|x|\rightarrow\infty}d(x)$.
Then the following hold.
\begin{itemize}
\item[(i)] Suppose that $\zeta<0$, and either \eqref{E2.5} or \eqref{E2.6} holds.
Then $F$ satisfies the $+$MP if and only if $\lambda_{1}^{+}(F)>0$.
\item[(ii)] Suppose that $\xi>0$, and either \eqref{E2.5} or \eqref{E2.6} holds. Furthermore, assume
\eqref{ET3.7A}. Then $F$ satisfies the $-$MP if and only if $\lambda_{1}^{-}(F)>0$.
\end{itemize}
\end{proposition}

We need a small lemma to prove \cref{P3.1}.
\begin{lemma}\label{L3.4}
The following hold for any smooth domain $\Omega$. 
\begin{itemize}
\item[(i)] $-\sup_{\Omega}\ c(x)\,\leq\, \inf_{\Omega} d(x)$.
\item[(ii)] $-\sup_{\Omega} \ c(x)\,\leq\, \pplamplus(F,\Omega)$.
\item[(iii)] $\inf_{\Omega}\ d(x)\,\leq\, \pplamminus(F,\Omega)$.
\end{itemize}
\end{lemma}
	
\begin{proof}
Part (i) follows from convexity property of $F$. Note that for $\lambda=-\sup_{\Omega} \ c(x)$, $\psi=1$
is an admissible function for $\pplamplus(F,\Omega)$. This gives us (ii).
In a similar fashion we get (iii).
\end{proof}

Now we prove \cref{P3.1}
\begin{proof}[Proof of \cref{P3.1}]
First consider (i). Assume that $\lambda_{1}^{+}(F)>0$. Using \cref{L3.4} we obtain 
\begin{align}\label{EP3.1A}
0\,<\,-\zeta\,=\,\lim_{r\rightarrow\infty}\left(-\sup_{\bar{\sB}_r^c}\ c(x)\right)\leq\,\lim_{r\rightarrow\infty}\,\pplamplus(F,\overline{\sB}_r^c )\,.
\end{align}
By \cref{T3.6}, we obtain $\pplamplus(F)>0$, and therefore, using \cref{T3.3}
 we see that $F$ satisfies the $+$MP. 
To show the converse direction we assume that $F$ satisfies $+$MP. 
Then \cref{T2.5} implies that $\plamplus(F)\geq 0$. Using \cref{T3.2} we then have
$\lambda_{1}^{+}(F)\geq 0$. If possible, suppose that $\lambda_{1}^{+}(F)= 0$. 
We show that there exists a bounded principal eigenfunction $\varphi$
which would give a contradiction to the validity of $+$MP, and hence we must have
$\lambda_{1}^{+}(F)> 0$. Consider a smooth positive function $\phi$ satisfying 
$\phi=1$ in $\sB^c_r$ for some large $r$. Since $\zeta<0$, we have a compact set $K$ satisfying 
$$ c(x)<\frac{\zeta}{2} \phi(x), \quad \phi(x)=1, \quad x\in K^c.$$
Recall the Dirichlet principal eigenfunciton $\psi^+_{1,n}$ from \eqref{EL3.1A}.
Choose $\kappa_n=\max_{\sB_n}\frac{\psi^+_{1,n}}{\phi}$ and define 
$v_n=\kappa^{-1}_n\psi^+_{1, n}$. Observe that $\phi-v_n$ must vanish in $K$. Indeed,
in $\sB_n\setminus K$ we have 
\begin{align*}
\cM_{\lambda,\Lambda}^-(x, D^2(\phi-v_n))-\gamma(x)|D\phi-v_n|-\delta(x)(\phi-v_n)
&\leq F(0, 0, \phi, x)- F(D^2 v_n, Dv_n, v_n, x)
\\
&\leq \frac{\zeta}{2} + \lambda^+_{1,n} v_n
\\
&\leq \frac{\zeta}{2} + \lambda^+_{1,n} <0\,,
\end{align*}
for all large $n$,
and therefore, by strong maximum principle \cite[Lemma~3.1]{QS08}, $\phi-v_n$ can not vanish in $\sB_n\setminus K$. Now applying Harnack's inequality and standard 
$\Sob^{2, p}$ estimates we can extract a convergent subsequence of $v_n$ converging to a positive eigenfunction $\varphi$. This completes the proof.


The proof for (ii) would be analogous.
\end{proof}

Next we prove \cref{T2.3}
\begin{proof}[Proof of \cref{T2.3}]
(i)\, From the definition it follows that 
$$
\pplamplus(F,\bar{\sB}_r^c )\geq \pplamplus(\tilde{F},\overline{\sB}_r^c )
-\sup_{\bar{\sB}_r^c} \tilde\gamma(x)\,,
$$ 
and then letting $r$ towards infinity we have 
$$\lim_{r\rightarrow\infty}\pplamplus(F,\bar{\sB}_r^c )\,\geq\,
 \lim_{r\rightarrow\infty}\pplamplus(\tilde{F},\overline{\sB}_r^c )\geq \pplamplus(\tilde{F})
 =\lambda_{1}^{+}(\tilde{F})\,.
 $$
Since $\tilde\gamma(x)\geq 0$, it gives us 
$\lambda_{1}^{+}(\tilde{F})\geq \lambda_{1}^{+}(F)$. Combining it with above calculation, we find 
$$\lim_{r\rightarrow\infty}\pplamplus(F,\bar{\sB}_r^c )\geq \lambda_{1}^{+}(F).$$ Applying
\cref{T3.6} we obtain  $\lambda_{1}^{+}(F)=\pplamplus(F)$.

(ii)\, Using \cref{L3.4} and the given hypothesis we find 
\begin{align*}
\lambda_{1}^{+}(F)\,\leq\,-\limsup_{|x|\rightarrow\infty}\ c(x)\,=\,
\lim_{r\rightarrow\infty}\ \big(-\sup_{\overline{\sB}_r^c}\ c(x) \big)
\,\leq\,\lim_{r\rightarrow\infty}\pplamplus(F,\overline{\sB}_r^c )\,.
\end{align*}
Hence, by \cref{T3.6}, we get $\lambda_{1}^{+}(F)=\pplamplus(F)$.

(iii)\, We show that under the given condition we have (ii). Hence it is enough to show that 
if $\sigma<\limsup_{|x|\rightarrow\infty}\ c(x)$ then $\lambda_{1}^{+}(F)\leq -\sigma$.
Now define a positive function 
$$\psi(x)=\exp\left(-\frac{1}{1-|\varepsilon x|^2}\right)$$
on the ball $\sB_{\frac{1}{\varepsilon}}$ where an appropriate $\varepsilon$ will be chosen later.
It is easily checked that 
\begin{align*}
D_{x_i}\psi &=\,  \frac{-2\varepsilon^2x_i}{(1-|\varepsilon x|^2)^2}\psi,
\\
D_{x_ix_j}\psi &=\, \bigg[\frac{4\varepsilon^4}{(1-|\varepsilon x|^2)^4}x_ix_j - \frac{2\varepsilon^2}{(1-|\varepsilon x|^2)^2}\delta_{ij}-\frac{8\varepsilon^4}{(1-|\varepsilon x|^2)^3}x_ix_j\bigg]\psi.
\end{align*}
For $x_0\in\RN$, define $\phi(x)=\psi(x-x_0)$.		
We will choose $\varepsilon$ and $x_0$ such that
\begin{equation}\label{ET2.3A}
F(D^2\phi,D\phi,\phi,x)-\sigma\phi>0\quad \text{in}\; \sB_{\frac{1}{\varepsilon}}(x_0).
\end{equation}
Since all the notions of eigenvalues of $F$ coincide in bounded domains (cf.\ \cite{QS08}),
 using \eqref{ET2.3A} we deduce
$$ -\sigma \geq\, \plamplus(F,\sB_{\frac{1}{\varepsilon}}(x_0)) \,=\,
\lambda^{+}_1(F,\sB_{\frac{1}{\varepsilon}}(x_0))\geq \lambda_{1}^{+}(F).$$
Thus we only need to establish \eqref{ET2.3A}. For a different way to construct such subsolutions we refer
\cite{R08}. Using (H3) we see that
\begin{align}\label{ET2.3B}
F(D^2\phi,D\phi,\phi,x)-\sigma\phi 
&=\, F(D^2\phi,D\phi,\phi,x)-F(0,0,\phi,x) + F(0,0,1,x)\phi-\sigma \phi \nonumber 
\\
& \geq\, \cM_{\lambda,\Lambda}^-(x,D^2\phi)-\gamma(x)|D\phi|+c(x)\phi -\sigma\phi \nonumber
\\
& \geq \bigg[\frac{4\lambda_0 \varepsilon^2 |\varepsilon (x-x_0)|^2}{(1-|\varepsilon (x-x_0)|^2)^4}-\frac{2N\Lambda_0\varepsilon^2}{(1-|\varepsilon (x-x_0)|^2)^2} -
\frac{8\Lambda\varepsilon^2 |\varepsilon (x-x_0)|^2}{(1-|\varepsilon (x-x_0)|^2)^3}\nonumber
\\
&\qquad - \frac{2\epsilon^2|x-x_0|\gamma(x)}{(1-|\epsilon (x-x_0)|^2)^2}+c(x)-\sigma \bigg]\phi\,.
\end{align}
Given $\varepsilon$ we choose $R$ such that $|\gamma(x)|\leq \varepsilon$ for $|x|\geq R$ and
then choose $x_0\in\RN$ satisfying $|x_0|\geq R+2\varepsilon^{-1}$. Furthermore, due to our
hypothesis, we can choose $x_0$ such that 
\begin{equation}\label{ET2.3C}
\inf_{\sB_{\frac{1}{\varepsilon}}(x_0)} c(x)\,>\,\sigma\,.
\end{equation}
We now compute \eqref{ET2.3B} in two steps.

\noindent{\bf Step 1.}\, Suppose $1-\delta<|\varepsilon(x-x_0)|^2<1$ where $\delta$ is very close to zero and will be chosen later. It then follows from \eqref{ET2.3B}  that
\begin{align*}
F(D^2\phi,D\phi,\phi,x)-\sigma\phi 
& \geq\, \frac{\varepsilon^2}{(1-|\varepsilon (x-x_0)|^2)^4}\bigg[4\lambda  (1-\delta)-2N\Lambda\delta^2 -8\Lambda(1-\delta)\delta  - 2 \delta^2\bigg]\phi
\\
&\qquad + \big(c(x)-\sigma \big)\phi\,.
\end{align*} 
Now we can choose small positive $\delta$, independent of $\varepsilon$, so that 
$$4\lambda  (1-\delta)-2N\Lambda\delta^2 -8\Lambda(1-\delta)\delta  - 2 \delta^2>0.$$
This proves \eqref{ET2.3A} in the annulus.

\noindent{\bf Step 2.}\, Now we are left with the part $0 \leq|\varepsilon (x-x_0)|^2\leq 1-\delta$ where $\delta$ is already chosen in Step 1. An easy calculation reveals
$$F(D^2\phi,D\phi,\phi,x)-\sigma\phi 
\,\geq\, \bigg[\big( c(x)-\sigma \big) -\frac{2N\Lambda\varepsilon^2}{\delta^2} -\frac{8\Lambda(1-\delta)\varepsilon^2}{\delta^3}  - \frac{2 \varepsilon^2}{\delta^2}\bigg]\phi\,.$$
Using \eqref{ET2.3C}, we can choose $\varepsilon$ small enough so that the RHS becomes positive.

Combining the above steps we obtain \eqref{ET2.3A}, completing the proof of part (iii).

(iv)\, This follows from \cref{T3.6}. Let us also provide a more direct proof.
Let $\varphi^*$ be an eigenfunction corresponding to $\lambda^+_1(F)=\lambda^+_1$. For $\delta, 
\varepsilon>0$ we define $\phi_\varepsilon=\varphi^*+\varepsilon V$. Choose $\varepsilon$ small
enough so that 
\begin{equation}\label{ET2.3D}
\delta\,\min_{\bar\sB}\varphi^*\;>\;\varepsilon \max_{\bar\sB} [F(D^2V,DV, V, x)+\lambda^+_1 V].
\end{equation}
By using convexity and homogeneity it follows that
\begin{align*}
F(D^2\phi_\varepsilon, D\phi_\varepsilon,\phi_\varepsilon, x)
&\leq F(D^2\varphi^*, D\varphi^*,\varphi^*, x) + \varepsilon F(D^2 V, D V,V, x)
\\
&= -\lambda^+_1 \varphi^* + \varepsilon \Ind_{\sB}(x) F(D^2V,DV, V, x) -
\varepsilon\lambda^+_1\,\Ind_{\sB^c}(x) V(x)
\\
&\leq -\lambda^+_1\phi_\varepsilon + \varepsilon \max_{\bar\sB} [F(D^2V,DV, V, x) + \lambda^+_1V]
\\
&\leq -(\lambda^+_1-\delta)\phi_\varepsilon,
\end{align*}
using \eqref{ET2.3D}. Hence $\pplamplus(F)\geq \lambda^+_1(F)-\delta$ and from the arbitrariness of
$\delta$ the result follows.
\end{proof}
Thus it remains to prove \cref{T2.6,T2.7}. Let us first attack \cref{T2.6}.
\begin{proof}[Proof of \cref{T2.6}]
Without any loss of generality, we assume that $\lambda^+_1(F)=0$.
Recall from \cref{L3.1} that the pair $(\psi^+_{1,n}, \lambda^+_{1,n})$ 
solving the Dirichlet eigenvalue problem with positive eigenfunction in $\sB_n$. That is,
\begin{equation}\label{ET2.6B}
F(D^2\psi^+_{1,n},D\psi_{1,n}^+,\psi_{1,n}^+,x)\,=\,-\lambda_{1,n}^+\psi_{1,n}^+\quad \text{in}\; \sB_n\,,
\quad \psi^+_{1,n}>0\; \text{in}\; \sB_n, \; \text{and}\; \psi^+_{1,n}=0\; \text{on}\; \partial\sB_n\,.
\end{equation}
Let $\kappa_n>0$ be such that 
$\kappa_n\psi^+_{1,n}\leq V$ in $\sB_n$ and it touches $V$ at some point in $\sB_n$. We claim that
$\kappa_n\psi^+_{1,n}$ has to touch $V$ inside $K$. Note that, by (H3), if $w=V-\kappa_n\psi^+_{1,n}$ then
$$\cM^-_{\lambda, \Lambda}(x, w)-\gamma |Dw|-\delta w\leq -\varepsilon V+\lambda^{+}_{1,n}(\kappa_n\psi^+_{1,n})\leq (-\varepsilon+\lambda^{+}_{1,n}) (\kappa_n\psi^+_{1,n})\,\leq\, 0
\quad \text{in}\; K^c\cap \sB_n\,,$$
for large $n$, using \eqref{ET2.6A} and \eqref{ET2.6B}.
Thus, if $w$ vanishes in $K^c\cap \sB_n$, then it must be identically 
$0$ in $K^c\cap \sB_n$, by the strong maximum principle \cite[Lemma~3.1]{QS08}. Bnd this is not possible
since $w>0$ on $\partial\sB_n$.
Now onwards we denote $\kappa_n\psi^+_{1,n}$ by $\psi^+_{1,n}$. By the above normalization, $\psi^+_{1,n}$
would converge, up to a subsequence, to a positive function $\varphi\in\Sobl^{2,p}(\RN), p<\infty$, an eigenfunction corresponding to $\lambda^+_1(F)=0$. See for instance, the argument in \cref{L3.1}.

We now show that any other principal eigenfunction is a multiple to $\varphi$.
For $\eta$, a small positive number, we define $\Xi_\eta=\psi^+_{1,n}-\eta V$. Using convexity
of $F$ we note that, in $\sB_n\cap K^c$,
\begin{align*}
F(D^2\Xi_\eta, D\Xi_\eta, \Xi, x) &\geq F(D^2\psi^+_{1,n},D\psi_{1,n}^+,\psi_{1,n}^+,x)
- \eta F(D^2V,DV,V,x)
\\
& \geq (-\lambda^{+}_{1,n}\psi^+_{1,n} + \eta\varepsilon V)
\\
&\geq (-\lambda^{+}_{1,n}+ \eta\varepsilon)V>0\,,
\end{align*}
provided we choose $n$ large (depending on $\eta$).
Let $\psi$ be any principal eigenfunction satisfying $$F(D^2\psi,D\psi,\psi,x)=0 \text{ in } \RN.$$
 Define
$$\delta=\delta(\eta)=\min_{K}\,\frac{\psi}{\Xi_\eta}.$$
Then $\delta\Xi_\eta\leq \psi$ on $K$. Since, by the Harnack inequality,
$$0< \inf_n\,\inf_{K}\psi^+_{1,n}\,\leq\, \sup_n\,\sup_{K}\psi^+_{1,n}<\,\infty,$$
we can choose $\eta_0$ small enough (independent of $n$) so that 
$$0< \inf_{\eta\in(0, \eta_0]}\,\inf_n\,\,\inf_{K}\Xi_\eta\,\leq\, \sup_{\eta\in(0, \eta_0]}\,\sup_n
\,\sup_{K}\Xi_\eta<\,\infty,$$
Thus, $\delta$ remains bounded and positive as $n\to\infty$ and $\eta\to 0$.
Since $F(D^2\psi,D\psi,\psi,x)=0$ in $\sB_n\cap K^c$ and $\lambda^+_1(F, \sB_n\cap K^c)>0$, it follows
from \cite[Theorem~1.5]{QS08}, that 
$$\delta \Xi_\eta\leq \psi\quad \text{in}\; \sB_n\,.$$
Furthermore, there exists $x_\eta\in K$ so that $\delta\Xi_\eta(x_\eta)=\psi(x_\eta)$. 
Now letting $n\to\infty$ first, and then $\eta\to 0$, we can extract a 
subsequence so that $\delta\to \theta>0$, and $x_\eta\to \hat{x}\in K$ 
 and $\theta\varphi(\hat{x})=\psi(\hat{x})$ with
$\theta\varphi\leq \psi$ in $\RN$. Let $u=\psi-\theta\varphi$. It is easy to see that
$$\cM^-_{\lambda, \Lambda}(x, u)-\gamma |Du|-\delta u\leq 0\quad \text{in}\; \RN.$$
By the strong maximum principle we must have $u=0$ and hence the proof.
\end{proof}

Finally, we prove \cref{T2.7}.
\begin{proof}[Proof of \cref{T2.7}]
The main idea of the proof is the same as that of the proof of \cref{T2.6}. Without any loss in generality, we
assume that $\lambda^-_1(F)=0$. 
Let $(\psi^-_{1, n}, \lambda^-_{1,n})$ be the pair satisfying the Dirichlet eigenvalue problem in
the ball $\sB_n$ i.e.,
\begin{equation}\label{ET2.7C}
F(D^2\psi^-_{1,n},D\psi_{1,n}^-,\psi_{1,n}^-,x)\,=\,-\lambda_{1,n}^-\psi_{1,n}^-\quad \text{in}\; \sB_n\,,
\quad \psi^-_{1,n}<0\; \text{in}\; \sB_n, \; \text{and}\; \psi^-_{1,n}=0\; \text{on}\; \partial\sB_n\,.
\end{equation}
By \cref{L3.2}, $\psi_{1,n}^-\searrow 0$ as $n\to\infty$. Recall that $G(M,p,u,x)\df-F(-M,-p,-u,x)$.
Denote by $\phi_n= -\psi_{1,n}^-$. Then we get from \eqref{ET2.7C} that
\begin{equation}\label{ET2.7D}
G(D^2\phi_{n},D\phi_{n},\phi_{n},x)\,=\,-\lambda_{1,n}^-\phi_{n}\quad \text{in}\; \sB_n\,,
\quad \psi_{n}<0\; \text{in}\; \sB_n, \; \text{and}\; \phi_{n}=0\; \text{on}\; \partial\sB_n\,.
\end{equation}
Note that $G$ satisfies (H1), (H2) and (H3) but it is a concave operator. So need some extra care to apply
the proof of \cref{T2.6}. Since $F$ is convex it follows from \eqref{ET2.7B} that
\begin{equation}\label{ET2.7E}
G(D^2V, DV, V, x)\,\leq\,F(D^2V, DV, V, x)\leq -(\lambda_1^-(F)+\varepsilon) V\quad \text{for all}\; x\in K^c.
\end{equation}
As done in \cref{T2.6}, using \eqref{ET2.7E}, we can normalize $\phi_n$ to touch $V$ from below and it would touch $V$ somewhere
in $K$. Therefore, we can apply the Harnack inequality (see \cref{L3.2}) to find a positive function $\varphi$ such that
$\phi_n\to\varphi$ in $\Sobl^{2, p}(\RN), p>N$, along some subsequence and 
$$0=-\lambda^-_1(F)\varphi=G(D^2\varphi,D\varphi,\varphi,x)=-F(-D^2\varphi,-D\varphi,-\varphi,x)\quad \text{in}\; \RN\,.$$
It is enough to show that $\varphi$ agrees with any other positive eigenfunction (up to a multiplicative constant) of $G$ with eigenvalue $0$. 

Next we define $\Xi_\eta(x)=\phi_n-\eta V$. Since $\norm{\phi_n-\varphi}_{L^\infty(K)}\to 0$, it is evident
that $\Xi_\eta>0$ for all $\eta$ small, independent of $n$. Using \eqref{ET2.7B} and \eqref{ET2.7C}, we see that, in
$K^c\cap\sB_n$,
\begin{align}\label{ET2.7F}
F(-D^2\Xi_\eta, -D\Xi_\eta, -\Xi, x) &\leq F(-D^2\phi_{n}, -D\phi_n, -\phi_n,x)+ \eta F(D^2V, DV, V, x)\nonumber
\\
&\leq (\lambda^{-}_{1,n}\phi_{n}-\eta\varepsilon V)\nonumber
\\
&\leq (|\lambda^{-}_{1,n}|- \eta\varepsilon)V<0,
\end{align}
for all large $n$. Now consider any positive eigenfunction $\psi\in\Sobl^{2,p}(\RN)$ satisfying
$$F(-D^2\psi, -D\psi, -\psi, x)\,=\,0\,,$$
and let 
$$\delta=\delta(\eta)=\min_{K}\frac{\psi}{\Xi_\eta}\,.$$
Then $-\delta\Xi_\eta\geq -\psi$ on $\partial K\cup\partial\sB_n$ for all $n$.
From \eqref{ET2.7B} if follows that $\lambda^+_1(F, K^c)\geq \varepsilon$.
Since $$\lambda^+_1(F, K^c\cap\sB_n)\to \lambda^+_1(F, K^c)>0 \text{ as } n\to\infty,$$ we can
apply the maximum principle \cite[Theorem~1.5]{QS08} in $\sB^c\cap K$ for all large $n$. From \eqref{ET2.7F} we
therefore
get $\psi\geq \delta\Xi_\eta$ and $\delta\Xi_\eta$ touches $\psi$ at some point in $K$.
Now we can follow the arguments in \cref{T2.6} we show that $\varphi=t\psi$ for some $t>0$.
Hence the proof.
\end{proof}

We conclude the paper with a remark on the eigenvalue problem in a general smooth unbounded domain.
\begin{remark}\label{R3.1}
For the case of an unbounded domain with smooth boundary all the results developed
here hold true and the proofs would be somewhat similar. As mentioned in \cite{BR15}, in case
of general unbounded domains, one needs the boundary Harnack property to control the behaviour of
eigenfunctions near the boundary. 
For the operator $F$, the boundary Harnack property has been obtained recently by Armstrong, Sirakov
and Smart in \cite[Appendix~A]{ASS}. Therefore one can easily adopt the techniques of \cite{BR15}
along with our results to deal with general unbounded domains.
\end{remark}

\section*{Acknowledgements}
The research of Anup Biswas was supported in part by DST-SERB grants EMR/2016/004810 and MTR/2018/000028.
Prasun Roychowdhury was supported in part by Council of Scientific \& Industrial Research (File no.  09/936(0182)/2017-EMR-I).

%

\end{document}